\documentclass[11pt]{amsart}
\usepackage{amssymb}
\usepackage{amscd}
\usepackage{amsmath}
\usepackage{amsmath}  
\usepackage[all]{xy}
\usepackage{mathrsfs}
\usepackage{graphicx}
 \usepackage{color}

\theoremstyle{plain}
\newtheorem{theorem}{Theorem}[section]
\newtheorem{lemma}[theorem]{Lemma}

\newtheorem{proposition}[theorem]{Proposition}
\newtheorem{definition}[theorem]{Definition}
\newtheorem{hypothesis}[theorem]{Hypothesis}
\newtheorem{example}[theorem]{Example}

\newtheorem{remark}[theorem]{Remark}

  1

\newcommand{\eins}{\boldsymbol{1}}

\DeclareMathOperator{\Ext}{Ext}
\DeclareMathOperator{\Gal}{Gal}

\DeclareMathOperator{\Hom}{Hom}

\DeclareMathOperator{\im}{im}

\newcommand{\CC}{\mathbb{C}}

\newcommand{\GG}{\mathbb{G}}
\newcommand{\NN}{\mathbb{N}}
\newcommand{\QQ}{\mathbb{Q}}
\newcommand{\Q}{\mathbb{Q}}
\newcommand{\RR}{\mathbb{R}}

\newcommand{\ZZ}{\mathbb{Z}}
\newcommand{\Z}{\mathbb{Z}}

\newcommand{\calF}{\mathcal{F}}

\newcommand{\calP}{\mathcal{P}}

\newcommand{\co}{\mathcal{O}}

\newcommand{\A}{\mathcal{A}}
\newcommand{\B}{\mathcal{B}}

\newcommand{\W}{\mathcal{W}}

\newcommand{\frp}{\mathfrak{p}}
\newcommand{\p}{\mathfrak{p}}

\newcommand{\cok}{\mathrm{cok}}

\newcommand{\Fit}{\mathrm{Fit}}

\newcommand{\bz}{\mathbb{Z}}
\newcommand{\La}{\Lambda}

\newcommand{\D}{\mathrm{d}}

\begin{document}

\title[On Weil-Stark elements]{On Weil-Stark elements, I: \\
 general properties}

\author{David Burns, Daniel Macias Castillo and Soogil Seo}

\begin{abstract} We construct a canonical family of elements in the reduced exterior power lattices of the unit groups of global fields. We prove that this family recovers the theory of cyclotomic elements in real abelian fields and also establish detailed arithmetic properties of its elements in the general case.\end{abstract}

\address{King's College London,
Department of Mathematics,
London WC2R 2LS,
U.K.}
\email{david.burns@kcl.ac.uk}

\address{Departamento de Matem\'aticas, 
Universidad Aut\'onoma de Madrid, 28049 Madrid (Spain);
and Instituto de Ciencias Matem\'aticas, 28049 Madrid (Spain).}
\email{daniel.macias@icmat.es}

\address{Yonsei University, Department of Mathematics, Seoul, Korea.}
\email{sgseo@yonsei.ac.kr}

\thanks{MSC: 11R27, 11R29 (primary), 11R34, 11R33 (secondary).}

\maketitle
%
\section{Introduction}

Our primary interest is the study of refined versions of Stark's seminal conjectures on the algebraic properties of the values at zero of Artin $L$-series over global fields. The results in the present article provide necessary preliminaries for this and, at the same time, contribute to the classical problem of identifying explicit elements in the unit groups of number fields.    

More precisely, we shall construct a canonical family of elements in the Galois-equivariant reduced exterior power lattices of unit groups of global fields. We refer to these elements as `Weil-Stark elements' since a variant (known to exist unconditionally) of Lichtenbaum's conjectural theory of Weil-\'etale cohomology \cite{lichtenbaum1} plays a  key role in their construction.

We establish some of the detailed arithmetic properties of these elements and also prove that, in the case of real abelian extensions of $\Q$, they recover the classical theory of cyclotomic elements. 
Then, in a sequel to this article, we show that the theory of Weil-Stark elements combines with the formalism of Tamagawa number conjectures that originates with Bloch and Kato to give a new, and essentially axiomatic, approach to the formulation of refinements of Stark's conjectures along with a range of unconditional evidence in support of 
the resulting predictions.

In a little more detail, the main contents of this article are as follows. In \S\ref{LambdaFAA} we review algebraic constructions and results of Sano and the first author that underlie our approach. In \S\ref{ords and bims} we study `locally-Gorenstein' orders and introduce bimodules and functors that are relevant to our theory. In \S\ref{We} we introduce Weil-Stark elements, with the general construction presented in Theorem \ref{main result} and the link to cyclotomic elements established in Theorem \ref{cyclo-weil thm}. Then in \S\ref{sgp} we prove that, over locally-Gorenstein orders, Weil-Stark elements have good `integrality' properties and, in particular, encode detailed structural information about ideal class groups (see Theorem \ref{main result2}). 
\medskip

\noindent{}{\bf Acknowledgements} We warmly thank Alexandre Daoud, Henri Johnston, Andreas Nickel, Takamichi Sano and especially an anonymous referee for some very helpful comments on earlier versions of the material presented here.  
The second author acknowledges support for this article as part of Grants CEX2019-000904-S, PID2019-108936GB-C21 and PID2022-142024NB-I00 funded by MCIN/AEI/ 10.13039/501100011033.

\section{Non-commutative algebra}\label{LambdaFAA} 

For the reader's convenience, in this first section we quickly review relevant aspects of constructions and results from \cite{bses}. To do so we write $\zeta(R)$ for the centre of a ring $R$ and, unless explicitly stated otherwise, we use `$R$-module' to mean `left $R$-module'. 

We fix a Dedekind domain $\Lambda$ that is a subring of $\CC$ and write $\mathcal{F}$ for its field of fractions. We let $\A$ be a $\Lambda$-order in a finite-dimensional separable $\mathcal{F}$-algebra $A$. We set $A_\CC:=\CC\otimes_{\mathcal{F}}A$. Given an $\A$-module $M$ we write $M_{\rm tor}$ for its $\Lambda$-torsion submodule and $M_{\rm tf}$ for its quotient by $M_{\rm tor}$. For any field extension $E$ of $\mathcal{F}$ we set $E\cdot M:=E\otimes_{\Lambda}M$.



We write $\Lambda_{(\frp)}$ for the localisation of $\Lambda$ at $\frp\in {\rm Spec}(\Lambda)$ and, for an $\A$-module $M$, set $M_{(\frp)}:= \Lambda_{(\frp)}\otimes_\Lambda M$. 
 A finitely generated $\A$-module $M$ is `locally-free' if, for every $\frp\in {\rm Spec}(\Lambda)$, the $\A_{(\frp)}$-module $M_{(\frp)}$ is free. In this case, the $\A_{(\frp)}$-rank ${\rm rk}_\A (M)$ of $M_{(\frp)}$ is independent of $\frp$, and $M$ is said to be `invertible' if ${\rm rk}_\A (M)=1$. We often use the fact that a locally-free $\A$-module is projective (by \cite[Prop. (8.19), Vol. I]{curtisr}). 



\subsection{}\label{Whitehead section}For $\frp\in {\rm Spec}(\Lambda)$ we write $\xi(\A_{(\frp)})$ for the $\Lambda_{(\frp)}$-order comprising the $\A_{(\frp)}$-submodule of $\zeta(A)$ generated by  $\{ {\rm Nrd}_A(M): M\in \bigcup_{n\in\NN}M_n(\A_{(\frp)})\}$. It is proved in \cite[Lem. 3.2]{bses} that the `Whitehead order' 
 $\xi(\A):={\bigcap}_{\frp\in {\rm Spec}(\Lambda)} \xi(\A_{(\frp)})$ of $\A$ is a $\Lambda$-order in $\zeta(A)$, that $\xi(\A)=\A$ if $\A$ is commutative and that any surjective map of $\Lambda$-orders $\A\to\mathcal{B}$ induces, upon restriction, a surjective map of $\Lambda$-orders $\xi(\A)\to\xi(\mathcal{B})$.

In general, neither $\xi(\A)\subseteq \zeta(\A)$ nor $\zeta(\A)\subseteq \xi(\A)$, but \cite[Def. 3.6]{bses} gives an explicit ideal $\delta(\A)$ of  $\zeta(\A)$ (that is equal to $\xi(\A) = \A$ if $\A$ is commutative) with $\delta(\A)\cdot\xi(\A)=\delta(\A)$. 

The links between the order $\xi(\A)$ and ideal $\delta(\A)$ and slightly different constructions of Johnston and Nickel \cite{JN} are discussed in \cite[Rem. 3.3 and 3.8]{bses}.

\subsection{}\label{fi review} 
%
For a natural number $m$ we write $[m]$ for the set of integers $i$ with $1\leq i\leq m$ and set $[m]^\ast := \{0\}\cup [m]$.
 Let $M$ be a matrix in ${\rm M}_{d\times d'}(A)$ with $d \ge d'$. Let $\{b_i\}_{i\in[d]}$ be the standard basis of $\A^d$. Then for any $t \in [d']^\ast$ and any $\varphi = (\varphi_i)_{i \in [t]}$ in $\Hom_\A(\A^{d},\A)^t$ we write ${\rm Min}^{d'}_{\varphi}(M)$ for the set of all $d'\times d'$ minors of the matrices $M(J,\varphi)$ that are obtained from $M$ by choosing any $t$-tuple of integers $J = \{i_1,i_2,\cdots , i_t\}$ with $1\le i_1< i_2< \cdots < i_t\le d'$, and setting
\begin{equation*}\label{fitting matrix} M(J,\varphi)_{ij} := \begin{cases} \varphi_{a}(b_i), &\text{if $j = i_a$ for $a \in [t]$}\\
                            M_{ij}, &\text{otherwise.}\end{cases}\end{equation*}
For $a\geq0$, the `$a$-th (non-commutative) Fitting invariant of $M$' is the $\xi(\A)$-ideal
\[ {\rm Fit}_{\A}^a(M) := \xi(\A)\cdot \{{\rm Nrd}_{A}(N): N\in {\rm Min}^{d'}_{\varphi}(M), \, \varphi\in \Hom_{\A}(\A^{d},\A)^t, \, t \le a\}.\]
%

%
%

A `presentation' $h$ of $\A$-modules is an exact sequence of finitely generated $\A$-modules 
$P' \xrightarrow{\theta} P \xrightarrow{} Z \to 0$;
$h$ is said to be `locally-free', resp. `locally-quadratic', if $P$ and $P'$ are locally-free, resp. locally-free with ${\rm rk}_{\A}(P')={\rm rk}_{\A}(P)$. 
If $h$ is locally-free, then 
its $a$-th Fitting invariant is the $\xi(\A)$-ideal  
\[ {\rm Fit}_{\A}^a(h) := {\bigcap}_{\p\in {\rm Spec}(\Lambda)}{\rm Fit}_{\A_{(p)}}^a(M(\theta,\p)),\]
where $M(\theta,\p)$ the matrix of $\theta_{(\p)}$ with respect to a choice of $\A_{(\p)}$-bases of $P'_{(\mathfrak{p})}$ and $P_{(\mathfrak{p})}$. If $\A$ is commutative, then $\Fit^a_\A(h)$ coincides with the classical $a$-th Fitting ideal of the $\A$-module $Z$ (as discussed by Northcott \cite{North}). 

%
%
%
%
%

\subsection{}\label{Rubin section}The ring $A$ decomposes as a direct product ${\prod}_{i\in I}A_i$ over a finite index set $I$, in which each ring $A_i$ is simple. For each $i\in I$ we fix a splitting field $E_i$ for $A_i$ over $\zeta(A_i)$ and a simple $E_i\otimes_{\zeta(A_i)}A_i$-module $V_i$. These choices determine, for each non-negative integer $r$, an `$r$-th reduced exterior power functor' $M \mapsto {{\bigwedge}}^r_AM$ from finitely generated $A$-modules to finitely generated $\zeta(A)$-modules. This functor behaves well under scalar extension and, 
for $s\in [r]^\ast$, there exists a natural duality pairing
$${\bigwedge}_A^rM\times{\bigwedge}_{A^{\rm op}}^s\Hom_A(M,A)\to{\bigwedge}_A^{r-s}M, \quad (m,\varphi)\mapsto\varphi(m).$$

If one fixes ordered $E_i$-bases of the spaces $V_i$, then subsets $\{m_j\}_{j \in [r]}$ of $M$ and $\{\varphi_j\}_{j \in [r]}$ of $\Hom_A(M,A)$ determine elements $\wedge_{j=1}^{j=r}m_j$ of ${\bigwedge}_A^rM$ and $\wedge_{j=1}^{j=r}\varphi_j$ of ${\bigwedge}_{A^{\rm op}}^r\Hom_A(M,A)$. The element $(\wedge_{i=1}^{i=r}\varphi_i)(\wedge_{j=1}^{j=r}m_j)$ 
belongs to $\zeta(A)$ and 
depends only on $\{m_j\}_{j \in [r]}$ and $\{\varphi_j\}_{j \in [r]}$ (cf. \cite[Lem. 4.10]{bses}). If $N$ is a finitely generated $\A$-module, its `$r$-th reduced Rubin lattice'   
$${\bigcap}_\mathcal{A}^rN:=\{a\in{\bigwedge}_A^r(\mathcal{F}\otimes_{\Lambda}N):(\wedge_{i=1}^{i=r}\varphi_i)(a)\in\xi(\mathcal{A})\text{ for all }\varphi_i\in\Hom_\mathcal{A}(N,\mathcal{A})\}$$
is then a $\xi(\A)$-lattice 
that has a range of useful properties (cf. \cite[Th. 4.19]{bses}). 

\begin{remark}\label{normalization rem}{\em 
The above constructions depend, up to isomorphism, on the choice of (bases of) the modules $\{V_i\}_{i \in I}$ (cf. \cite[Rem. 4.4]{bses}). 
However, there are two cases that arise in our theory and for which this ambiguity can be removed (in a compatible fashion). \

\noindent{}(i) If $A$ is commutative (so $\xi(\A) = \A$), then, for each $i\in I$, one can take $1$ as the basis of $V_i=E_i=A_i$. In this case, the above constructions recover classical exterior powers and the `Rubin lattices' introduced in \cite{R} (cf. \cite[Rem. 4.18]{bses}).\

\noindent{}(ii) If $\Gamma$ is a finite group, then the choice of a representation $\Gamma\to{\rm GL}_{\chi(1)}(\QQ^c)$ for each irreducible complex character $\chi$ of $\Gamma$ fixes a canonical set of modules $\{V_i\}_{i \in I}$ for the algebra $A = E[\Gamma]$ together with a distinguished basis of each (cf. \cite[Rem. 4.9]{bses}).}\end{remark} 



\subsection{}\label{332} 
We write $D(\A)$ for the derived category of $\A$-modules and $D^{{\rm lf}}(\A)$ for the full triangulated subcategory generated by bounded complexes of finitely generated locally-free $\A$-modules. (In the case $\A=\ZZ[G]$, a theorem of Swan \cite{swan} (see also \cite[Prop. (8.19), Vol. I]{curtisr}) implies $D^{{\rm lf}}(\ZZ[G])$ coincides with the full subcategory of $D(\ZZ[G])$ comprising complexes that are perfect). We write $D^{\rm lf}(\A)_{\rm is}$ for the subcategory of $D^{\rm lf}(\A)$ in which morphisms are restricted to isomorphisms, and $\calP(\xi(\A))$ for the category of graded invertible $\xi(\A)$-modules (though, often, we omit reference to the grading). Then, for any choice of normalization of reduced exterior powers as above, \cite[Th. 5.2]{bses} constructs a `reduced determinant functor' 
$$\D_{\A}:D^{\rm lf}(\A)_{\rm is}\to\calP(\xi(\A))$$ that respects exact triangles in $D^{\rm lf}(\A)$ and, on every locally-free $\A$-module $M$, gives
\[ \D_{\A}( M[0]) = \bigl( {{{{\bigcap}}}}_{\A}^{{\rm rk}_\A(M)}M, {\rm rr}_A(M_F)\bigr)\]
where ${\rm rr}_A(M_F)$ is the `reduced $A$-rank' of $A\otimes_{\A}M$. Writing ${\rm Mod}(A)$ for the category of finitely generated $A$-modules, the methods of loc. cit. also construct associated functors

$$\D_A:D^{\rm lf}(A)_{\rm is}\to\calP(\zeta(A))\quad\text{and}\quad \D^\diamond_A:{\rm Mod}(A)\to\calP(\zeta(A)).$$

These functors are related via canonical `passage to cohomology' isomorphisms
\begin{equation}\label{ptc iso} \D_{A}(X) \cong \underset{i\in\bz}{\bigotimes}
\, \D_{A}^\diamond(H^i(X))^{(-1)^{i}}\end{equation}
in $\mathcal{P}(\zeta(A))$, and are also such that the diagram of scalar extension functors 
\begin{equation}\label{eofs}
\xymatrix{ D^{\rm lf}(A)_{\rm is}  \ar@{->}[r]^{\D_{A}}   &  \calP(\zeta(A))  \\
D^{\rm lf}(\A)_{\rm is}  \ar@{->}[u] \ar@{->}[r]^{\D_{\A}} &  \calP(\xi(\A))  \ar@{->}[u]} 
\end{equation}
commutes (cf. \cite[Prop. 5.17]{bses}).

\section{Orders and bimodules}\label{ords and bims} 
In this section we introduce the orders and bimodules that are relevant to our theory. To do this we fix data $\Lambda, \calF, \A$ and $A$ as in \S\ref{LambdaFAA}.



\subsection{Locally Gorenstein orders}\label{conditions sub}  

\begin{definition}\label{LG}{\em We say $\A$ is `locally-Gorenstein' if there exists a $\Lambda$-linear anti-involution $\iota_\mathcal{A}$ of $\mathcal{A}$ such that for all $\frp\in {\rm Spec}(\Lambda)$, the dual $\Hom_{\Lambda}(\mathcal{A},\Lambda)_{(\mathfrak{p})}$ is a free rank one $\mathcal{A}_{(\mathfrak{p})}$-module with respect to the action $(a\theta)(a') := \theta(\iota_\mathcal{A}(a)a')$, for $\theta$ in $\Hom_{\Lambda}(\mathcal{A},\Lambda)_{(\mathfrak{p})}$ and $a,a'$ in $\A_{(\frp)}$.
}\end{definition}

\begin{example}\label{standard exam}{\em The above condition is satisfied in each of the following cases. \

(i) $\mathcal{A}$ is commutative and for every prime $\frp\in {\rm Spec}(\Lambda)$, the $\Lambda_{(\frp)}$-order $\A_{(\frp)}$ is Gorenstein (relative to the identity anti-involution). (See \cite{jensenthorrup} for a useful survey of Gorenstein orders.) 

(ii) $\mathcal{A} = \Lambda[G]$ for a finite group $G$ and $\iota_{\Lambda[G]}$ is the $\Lambda$-linear anti-involution $\iota_\#$ that inverts elements of $G$ (cf. \cite[(10.29) Cor.]{curtisr}). 

(iii) In the setting of (ii), fix a cyclic subgroup $C$ of $G$ and an idempotent $e$ of $\calF[C]$ with the property that 
$\iota_\#(e) = e$. Then the commutative $\Lambda$-order $\Lambda[C]e$ is monogenic, and hence also Gorenstein (cf. \cite{jensenthorrup}), and so the order $\A=\Lambda[G]c = \Lambda[G]\otimes_{\Lambda[C]}\Lambda[C]e$ is locally-Gorenstein with respect to the anti-involution $\iota_\#$. 

(iv) $\A$ is an hereditary order that possesses a $\Lambda$-linear anti-involution (this follows as a consequence, for example, of \cite[(4.3) Prop.]{curtisr}).}
\end{example}


For later use we record several useful consequences of the condition in Definition \ref{LG}.

\begin{lemma}\label{useful props remark} If $\A$ is locally-Gorenstein, then all of the following claims are valid. 

\begin{itemize}
\item[(i)] The functor $M \mapsto {\rm Hom}_\mathcal{A}(M,\mathcal{A})$ is exact on the category of finitely generated $\Lambda$-torsion-free $\mathcal{A}$-modules.
\item[(ii)] For every finitely generated $\A$-module $M$ the group ${\rm Ext}^i_{\mathcal{A}}(M,\mathcal{A})$ vanishes 
 for all integers $i$ with $i\ge 2$.
\item[(iii)] The functor $M \mapsto \Hom_{\mathcal{A}}(M,A/\mathcal{A})$ is exact on the category of finite $\mathcal{A}$-modules.
\item[(iv)] There is a natural identification ${\rm Ext}^1_\mathcal{A}(M,\mathcal{A})\cong \Hom_{\mathcal{A}}(M_{\rm tor},A/\mathcal{A})$. 
\end{itemize}
\end{lemma}

\begin{proof} For each $\A$-module $M$ we set $M^\ast := {\rm Hom}_\Lambda(M,\Lambda)$ and we note that the natural map $M_{\rm tf} \to (M_{\rm tf})^{\ast \ast} = M^{\ast\ast}$ is bijective.  

Claim (i) is equivalent to asserting that ${\rm Ext}^1_\A(X_{\rm tf},\A)$ vanishes for every finitely generated $\A$-module $X$. To prove this we need to show that, for any such $X$, any short exact sequence of $\A$-modules of the form $0 \to \A \to M \to X_{\rm tf} \to 0$ is split. To check this, we note that the $\Lambda$-module $X_{\rm tf}$ is projective and hence that  the functor $M \mapsto M^*$ applies to the given exact sequence to give an exact sequence $0 \to X^* \to M^* \to \A^* \to 0$. We then note that the latter sequence splits (since Definition \ref{LG} implies the $\A$-module $\A^*$ is projective) and so gives an isomorphism $M^*\cong X^* \oplus \A^*$. Applying the functor $M\mapsto M^*$ to this isomorphism, and noting $X_{\rm tf}$ and $M$ are $\Lambda$-torsion-free, we therefore obtain an isomorphism of $\A$-modules 
\[ M\cong M^{\ast\ast} \cong (X^\ast \oplus \A^\ast)^\ast \cong X^{\ast\ast} \oplus \A^{\ast\ast} \cong X_{\rm tf}\oplus \A\]
which can be seen to imply that the original sequence is split, as required. 

To prove claim (ii) we fix an exact sequence of $\A$-modules $0 \to X \to P \to M \to 0$ with $P$ finitely generated and projective. Then, for each integer $i \ge 2$, the long exact ${\rm Ext}_\A^{\bullet}(-,\A)$-sequence associated to this sequence induces an 
isomorphism of abelian groups ${\rm Ext}^{i-1}_\A(X,\A) \cong {\rm Ext}^i_\A(M,\A)$. This reduces us proving ${\rm Ext}_\A^1(X,\A)$ vanishes for a finitely generated $\Lambda$-torsion-free $\A$-module $X$ and this follows from the proof of claim (i). 

To prove claim (iii) it suffices to show that ${\rm Ext}_\A^1(M,A/\mathcal{A})$ vanishes for every finite $\A$-module $M$. This is true since the first and third modules in the natural exact sequence 
\[ {\rm Ext}^1_\A(M,A) \to {\rm Ext}_\A^1(M,A/\mathcal{A})\to {\rm Ext}_\A^2(M,\mathcal{A})\]
vanish, the first as it is $\Lambda$-torsion $\mathcal{F}$-module and the third by claim (ii). 

Finally, to prove claim (iv) we use the exact sequence 
\[ {\rm Ext}_\A^1(M_{\rm tf},\mathcal{A})\to {\rm Ext}_\A^1(M,\mathcal{A}) \to {\rm Ext}_\A^1(M_{\rm tor},\mathcal{A}) \to {\rm Ext}_\A^2(M_{\rm tf},\mathcal{A})\]
that is induced by the tautological exact sequence $0 \to M_{\rm tor} \to M \to M_{\rm tf} \to 0$. We note that the first and last modules in the displayed vanish, respectively by claims (i) and (ii). The isomorphism in claim (iv) then follows from the fact that the exact sequence $0 \to \A \to A \to A/\A \to 0$ induces an isomorphism 
${\rm Ext}_\A^1(M_{\rm tor},\mathcal{A})\cong \Hom_\A(M_{\rm tor},A/\A)$. \end{proof} 

\subsection{Bimodules and functors}
%
\subsubsection{}
\label{bimodule} In the sequel we also assume to be given a finite group $G$ and a finitely generated $(\mathcal{A},\Lambda[G])$-bimodule $\Pi$ that satisfies the following hypothesis. 
\begin{hypothesis}\label{Pi hyp}\ {\em 
\begin{itemize}
\item[($\Pi_1$)] $\Pi$ is a locally-free $\mathcal{A}$-module.
 \item[($\Pi_2$)] $W\mapsto W\otimes_\mathcal{A}\Pi$ induces an injection from the set of isomorphism classes of simple right $A_\CC$-modules to the set of isomorphism classes of simple right $\CC[G]$-modules.
 \item[($\Pi_3$)] If $\A$ is locally-Gorenstein, then ($\Pi_2$) is also true with $\Pi$ replaced by $\check\Pi:= \Hom_{\mathcal{A}}(\Pi,\mathcal{A})$, regarded as an $(\mathcal{A},\Lambda[G])$-bimodule by means of the anti-involutions $\iota_\mathcal{A}$ and $\iota_\#$.
\end{itemize} }
\end{hypothesis}

In the sequel we write $\widehat G$ for the set of irreducible complex characters of $G$. 
\begin{remark}\label{wedd bij}{\em Under condition ($\Pi_2$) one obtains a bijection $\Pi_*$ from the set of Wedderburn components ${\rm Wed}_A$ of $A_\CC$ to a subset $\Upsilon_\Pi$ of $\widehat G$ and hence a commutative diagram 
\begin{equation}\label{key commute}\begin{CD}
K_1(\QQ[G]) @> {\rm Nrd}_{\QQ[G]}  >> \zeta(\QQ[G])^\times @> \subset  >> \zeta(\CC[G])^\times = {\prod}_{\widehat G}\CC^\times\\
@V \mu^1_\Pi VV @. @VV \iota_\Pi V\\
 K_1(A) @> {\rm Nrd}_A >> \zeta(A)^\times @>\subset  >> \zeta(A_\CC)^\times = {\prod}_{{\rm Wed}_A}\CC^\times.
\end{CD}\end{equation}
Here $K_1(A)$ is the $K_1$-group of the category of finitely generated $A$-modules, $\mu^1_\Pi$ sends the class of an automorphism $\alpha$ of a finitely generated left $\QQ[G]$-module $V$ to the class of the automorphism ${\rm id}_\Pi\otimes_{\ZZ[G]}\alpha$ of $\Pi\otimes_{\ZZ[G]}V$ and $\iota_\Pi$ sends each element $(z_\chi)_{\chi\in \widehat G}$ to $(z_{\Pi_*(C)})_{C\in {\rm Wed}_A}$.}\end{remark}

\begin{example}\label{exam1}{\em We consider applications in which the following bimodules $\Pi$ arise.\

(i) If $\Lambda\subset\QQ$ and $A$ is a direct factor of $\QQ[G]$ then for any homomorphism of rings $\kappa:\Lambda[G] \to \mathcal{A}$  one can set $\Pi_\kappa := \mathcal{A}$. In all cases the lattice $\Pi_\kappa$ satisfies the conditions $(\Pi_1)$ and $(\Pi_2)$, and $\Upsilon_{\Pi_\kappa}$ is the subset of $\widehat{G}$ comprising characters which occur in $A_\CC$. The order $\mathcal{A}$ is locally-Gorenstein under any of the additional conditions discussed in Example \ref{standard exam} and also if $\mathcal{A}$ is a regular $\Lambda$-algebra of dimension one, and in all such cases $\Pi_\kappa$ satisfies $(\Pi_3)$.


(ii) Fix a representation $\psi: G \to {\rm GL}_{\psi(1)}(\mathcal{O}_\psi)$ of $G$ over a finite extension $\mathcal{O}_\psi$ of $\La$. 
Set $\mathcal{A}_\psi := \mathcal{O}_\psi$ and regard $\Pi_\psi := \mathcal{O}_\psi^{\psi(1)}$ as an $(\mathcal{O}_\psi,\Lambda[G])$-bimodule via $\psi$. Then $\mathcal{A}_\psi $ is locally-Gorenstein, $\Pi_\psi$ satisfies Hypothesis \ref{Pi hyp} and $\Upsilon_{\Pi_\psi} = \{\psi\}$.
}\end{example}

\subsubsection{}
The bimodule $\Pi$ gives rise to functors $M \mapsto {^\Pi}M$ and $M\mapsto {_\Pi}M$ from the category of left $G$-modules to the category of left $\mathcal{A}$-modules by setting
\[ {^\Pi}M := H^0(G,\Pi\otimes_\ZZ M)\,\,\,\text{ and }\,\,\,{_\Pi}M:= H_0(G,\Pi\otimes_\ZZ M) = \Pi\otimes_{\ZZ[G]}M,\]
where the left action of $G$ on the tensor product is via $g(\pi\otimes m) = (\pi)g^{-1}\otimes g(m)$. 

We write $T_{\Pi,M}: {_\Pi}M \to {^\Pi}M$ for the homomorphism of $\mathcal{A}$-modules that is induced by sending each element $\pi\otimes m$ of $\Pi\otimes_\ZZ M$ to ${\sum}_{g \in G}g(\pi\otimes m)$.

\begin{remark}\label{ct remark}{\em If the $G$-module $\Pi\otimes_\ZZ M$ is cohomologically-trivial, then the Tate cohomology groups $\hat H^{-1}(G,\Pi\otimes_\ZZ M) = \ker(T_{\Pi,M})$ and $\hat H^{0}(G,\Pi\otimes_\ZZ M) = \cok(T_{\Pi,M})$ vanish and so the map $T_{\Pi,M}$ is bijective.}\end{remark}

\begin{lemma}\label{trace sequences} Let $\phi:P^0\to P$ be a map between cohomologically-trivial $G$-modules. Then $T_{\Pi,P^0}$ is bijective and there exists a canonical exact sequence of $\A$-modules
$$0\to{^\Pi}\ker(\phi)\stackrel{\iota}{\longrightarrow}{_\Pi}P^0\stackrel{{_\Pi}\phi}{\longrightarrow}{_\Pi}P\stackrel{\varpi}{\longrightarrow}{_\Pi}\cok(\phi)\to 0$$
in which $\iota$ is the composite ${^\Pi}\ker(\phi)\to{^\Pi}P^0 \xrightarrow{T_{\Pi,P^0}^{-1}}{_\Pi}P$ and $\varpi$ is induced by the canonical projection map $P\to\cok(\phi)$.
\end{lemma}
\begin{proof} If $N$ is a cohomologically-trivial $G$-module, then, since $\Pi$ is $\Lambda$-torsion-free, the $G$-module $\Pi\otimes_\ZZ N$ is also cohomologically-trivial. The first claim therefore follows directly from Remark \ref{ct remark}. Next, we note that the functors $M \to {^{\Pi}}M$ and $M \to {_\Pi}M$ are respectively left and right exact, and hence that there is an exact commutative diagram of $\mathcal{A}$-modules

\begin{equation*}
\begin{CD}
0 @> >>  ^{\Pi}\ker(\phi) @>  >>  ^{\Pi}\! P^0 @> {^\Pi} \phi >> ^{\Pi}\! P \\
@. @. @A T_{\Pi, P^0} AA @AA T_{\Pi, P} A \\
@. @. _\Pi  P^0 @> {_\Pi}\phi >> _\Pi  P @> \varpi >> _\Pi \cok(\phi) @> >> 0
\end{CD}\end{equation*}
in which the vertical maps are bijective by Remark \ref{ct remark}. The claimed exact sequence follows directly from this diagram. 
\end{proof}

\section{Weil-Stark elements}\label{We}

In this section we present an unconditional construction of a canonical family of elements in the reduced Rubin lattices of unit groups. 
 
 To do so we fix a finite Galois extension $L/K$ of global fields and set
$G:= \Gal(L/K)$. We write $S^\infty_K$ for the set of archimedean places of $K$ (so $S_K^\infty=\emptyset$ unless $K$ is a number field) and fix a finite non-empty set of places $S$ of $K$ that contains $S^\infty_K$. For a set of places $\Sigma$ of $K$ 
 we write $\Sigma_L$ for the
set of places of $L$ lying above $\Sigma$, $Y_{L,\Sigma}$ for the free abelian group on $\Sigma_L$ and $X_{L,\Sigma}$ for the submodule of $Y_{L,\Sigma}$ comprising elements whose coefficients sum to zero. If $S^\infty_K \subseteq \Sigma$, we write $\co_{L,\Sigma}$ for the subring of $L$ comprising
elements integral at all places outside $\Sigma_L$ and $\co_{L,\Sigma}^\times$ for its unit group. The groups $X_{L,\Sigma}, Y_{L,\Sigma}, \co_{L,\Sigma}$ and $\co_{L,\Sigma}^\times$ are all stable under the natural action of $G$. 

We assume throughout that the constructions reviewed in \S\ref{Rubin section} and \S\ref{332} are, in all relevant cases, normalised as in Remark  \ref{normalization rem}. 


\subsection{The general construction}\label{construction section}

\subsubsection{}To help set the context of our construction we establish some useful facts concerning the reduced Rubin lattices of unit groups. 
%
%
%
%

For $\psi$ in $\widehat{G}$ we fix an associated complex representation $V_\psi$ of $G$ and set 
\[ r_S(\psi):= \begin{cases} {\sum}_{v \in S}{\rm dim}_\CC(H^0(G_v,V_\psi)), &\text{if $\psi \not= {\bf 1}$},\\
|S|-1, &\text{if $\psi = {\bf 1}$,}\end{cases}\]
where $G_v$ is the decomposition subgroup in $G$ of an arbitrary place of $L$ above $v$. We also define a primitive idempotent of $\zeta(\CC[G])$ by setting $e_\psi:=(\psi(1)/|G|){\sum}_{g\in G}\psi(g)g^{-1}$.

\begin{lemma}\label{fe} Fix a bimodule $\Pi$ that satisfies Hypothesis \ref{Pi hyp} and set $d := {\rm rk}_\mathcal{A}(\Pi)$. Then the following claims are valid for every non-negative integer $a$. 
\begin{itemize}
\item[(i)] For each $\psi$ in $\widehat{G}$ set $r(\psi) := r_S(\psi)$ and $a(\psi) := \psi(1)a$. Then one has 
\[ {\rm dim}_\CC\bigl(e_\psi\bigl( \CC\cdot {\bigcap}_{\ZZ[G]}^a\mathcal{O}_{L,S}^\times\bigr)\bigr) = 
\begin{cases} 0, &\text{if $a(\psi) > r(\psi)$},\\
\binom{r(\psi)}{a(\psi)}, &\text{if $a(\psi) \le r(\psi)$}.
\end{cases}
\]
\item[(ii)] There is an idempotent $e$ of $\zeta(\mathcal{F}[G])$ such that $\calF\cdot{{\bigcap}}_{\A}^{ad}(^{\Pi}\co^\times_{L,S}) \subseteq e\cdot\bigl(\calF\cdot{{{\bigcap}}}_{\ZZ[G]}^a\co^\times_{L,S}\bigr)$, with equality if the cokernel of the following restriction map is finite: 
\[ \Hom_G(\mathcal{O}^\times_{L,S},\ZZ[G]) \xrightarrow{\theta \mapsto 1\otimes \theta} \Hom_\A({^{\Pi}\mathcal{O}^\times_{L,S}},{{^\Pi}\ZZ[G]}).\]
\end{itemize}
\end{lemma}

\begin{proof} At the outset we note the Dirichlet regulator map induces an isomorphism of $\RR[G]$-modules $R_{L,S}:\RR\cdot\mathcal{O}_{L,S}^\times \cong \RR\cdot X_{L,S}$. This combines with the Noether-Deuring Theorem to imply the existence of a finite index $G$-submodule of $\mathcal{O}_{L,S}^\times$ that is isomorphic to $X_{L,S}$, and so it is enough to prove the stated claims after replacing $\mathcal{O}_{L,S}^\times$ by $X_{L,S}$. To do this we often use the general result of \cite[Th. 4.17(ii)]{bses} regarding reduced Rubin lattices (though do not always mention it). For example, as a first application, for each $\psi$ in $\widehat{G}$, this result implies 
\begin{equation}\label{explicit span} e_{\psi}\bigl( \CC\cdot {\bigcap}_{\ZZ[G]}^aX_{L,S}\bigr) = {\bigwedge}_\CC^{\psi(1)a} \bigl(V_{\psi}^\ast\otimes_{\ZZ[G]}X_{L,S}\bigr),\end{equation}
with $V_\psi^* := \Hom_\CC(V_\psi,\CC)$ (regarded as a right $G$-module in the obvious way).

To deduce claim (i) from this description, it is enough to show ${\rm dim}_\CC\bigl(V_\psi^\ast\otimes_{\ZZ[G]}X_{L,S}\bigr) = r_S(\psi)$. This can be verified by an easy computation that combines the exact sequence $0 \to X_{L,S}\to Y_{L,S} \to \ZZ \to 0$ with the fact $Y_{L,S}$ is isomorphic as a $G$-module to the direct sum ${\bigoplus}_{v \in S}(\ZZ[G]\otimes_{\ZZ[G_v]}\ZZ)$ and that for each $v \in S$ there is an isomorphism of $\CC$-spaces 
\[ V_\psi^\ast\otimes_{\ZZ[G]}\bigl(\ZZ[G]\otimes_{\ZZ[G_v]}\ZZ\bigr) \cong V_\psi^\ast\otimes_{\ZZ[G_v]}\ZZ \cong H^0(G_v,V_\psi)^\ast.\]

To prove claim (ii) we set $Z_{L,S} := {^{\Pi}X_{L,S}}$ and fix an isomorphism of $A$-modules $\mathcal{F}\cdot \Pi \cong A^d$. Then, under condition $(\Pi_2)$, for a simple right $A_\CC$-module $W$ the module $W^{\oplus d} = W\otimes_A A^d \cong W\otimes_\A \Pi$ is a simple right $\CC[G]$-module and hence isomorphic to $V_{\psi_W}^*$ for some ${\psi_W}$ in $\widehat{G}$. It follows that $\psi_W(1) = {\rm dim}_\CC(W)d$ and hence that the description (\ref{explicit span}) combines with the  corresponding description of ${\bigcap}_{\A}^{ad}Z_{L,S}$ to give identifications     
\begin{multline*} e_{\psi_W}\bigl( \CC\cdot {\bigcap}_{\ZZ[G]}^aX_{L,S}\bigr) = {\bigwedge}_\CC^{\psi_W(1)a} \bigl(V_{\psi_W}^\ast\otimes_{\ZZ[G]}X_{L,S}\bigr) = {\bigwedge}_\CC^{\psi_W(1)a} \bigl( W\otimes_\A (\Pi\otimes_{\ZZ[G]}X_{L,S})\bigr)\\
 = {\bigwedge}_\CC^{{\rm dim}_\CC(W)ad} \bigl(W\otimes_{\mathcal{A}}Z_{L,S}\bigr) = e_W\bigl(\CC\cdot {\bigcap}_\mathcal{A}^{ad}Z_{L,S}\bigr),\end{multline*}
where $e_W$ is the primitive central idempotent of $A_\CC$ corresponding to $W$. Via these identifications (for each $W$) one has  $\CC\cdot {\bigcap}_\mathcal{A}^{ad}Z_{L,S} = e\bigl(\CC\cdot {\bigcap}_{\ZZ[G]}^aX_{L,S})$, where $e$ is the sum of $e_{\psi_W}$ 
as $W$ ranges over the simple right $A_\CC$-modules for which $e_W\bigl(\CC\cdot {\bigcap}_\mathcal{A}^{ad}Z_{L,S}\bigr)\not= (0)$. In particular, the idempotent $e$ belongs to $\zeta(\mathcal{F}[G])$ since the corresponding sum ${\sum}_W e_W$ belongs to $\zeta(A)$. 

In a similar fashion, the composite restriction map 
\begin{multline*} \varrho: \mathcal{F}\cdot \Hom_G(X_{L,S},\ZZ[G]) \to \mathcal{F}\cdot  \Hom_\A(Z_{L,S},{{^\Pi}\ZZ[G]}) = \Hom_A(\mathcal{F}\cdot Z_{L,S},\mathcal{F}\cdot \Pi)\\ \cong \Hom_A(\mathcal{F}\cdot Z_{L,S},A^d) = \mathcal{F}\cdot\Hom_\A(Z_{L,S},\A)^d\end{multline*}
induces for each $W$ a homomorphism of $\CC$-spaces 
\[  {\bigwedge}_\CC^{\psi_W(1)a} \bigl(V_{\psi_W}\otimes_{\ZZ[G]^{\rm op}}\Hom_G(X_{L,S},\ZZ[G])\bigr) \to {\bigwedge}_\CC^{{\rm dim}_\CC(W)ad} \bigl(W^\ast\otimes_{\mathcal{A}^{\rm op}}(\Hom_\A(Z_{L,S},\A))\bigr),\]
and hence 
 of reduced exterior powers 

\[ e\bigl({\bigwedge}_{\mathcal{F}[G]}^a(\mathcal{F}\cdot \Hom_G(X_{L,S},\ZZ[G]))\bigr) \to {\bigwedge}_{A}^{ad}(\mathcal{F}\cdot \Hom_\A(Z_{L,S},\A)).\] 

Under these identifications, $e\bigl(\calF\cdot{{{\bigcap}}}_{\ZZ[G]}^a X_{L,S}\bigr)$ is equal to the 
subset of $\CC\cdot {\bigcap}_\mathcal{A}^{ad}Z_{L,S}$ comprising elements $x$ with the property $({\wedge}_{i=1}^{i=a}\varrho(\varphi_i))(x) \in \zeta(A)$ for all $\varphi_i \in \mathcal{F}\cdot\Hom_G(X_{L,S},\ZZ[G])$, and each element 
${\wedge}_{i=1}^{i=a}\varrho(\varphi_i)$ belongs to ${\bigwedge}_{A}^{ad}(\mathcal{F}\cdot \Hom_\A(Z_{L,S},\A))$. 

The inclusion in claim (ii) follows directly from these facts (and \cite[Th. 4.17(ii)]{bses}), and, in addition, this inclusion is an equality under the stated condition since then the map $\varrho$ is surjective and hence the elements ${\wedge}_{i=1}^{i=a}\varrho(\varphi_i)$ span ${\bigwedge}_{A}^{ad}(\mathcal{F}\cdot \Hom_\A(Z_{L,S},\A))$. This completes the proof of claim (ii). 
%
%
%
\end{proof}

\subsubsection{}\label{dataasin} In order to state the main result of this section (the proof of which will then occupy the rest of \S\ref{construction section}), we assume  $S$ contains all places that ramify in $L/K$. 
We also fix data $\Lambda,A,\mathcal{A}$ as in \S \ref{ords and bims} and a bimodule $\Pi$ that satisfies $(\Pi_1)$ and $(\Pi_2)$ in Hypothesis \ref{Pi hyp} for $G={\rm Gal}(L/K)$. (We do not, however, assume in this section that $\A$ is locally-Gorenstein.)  

We now also fix a surjective homomorphism of $\mathcal{A}$-modules 
\begin{equation}\label{choice of pi}\pi: {_{\Pi}}X_{L,S}\to Y_\pi\end{equation}
in which $Y_\pi$ is locally-free. We then define an idempotent of $\zeta(A)$ by setting
$e_\pi:={\sum} e,$
where the sum runs over all primitive idempotents $e$ of $\zeta(A)$ for which $e(\mathcal{F}\cdot\ker(\pi))=(0)$. 

%




For the Whitehead order $\xi(\A)$ of $\A$ (cf. \S\ref{Whitehead section}), we say a $\xi(\A)e_\pi$-module $M$ is `locally-cyclic' if for all $\frp\in{\rm Spec}(\Lambda)$, the $\xi(\A_{(\frp)})$-module $M_{(\frp)}$ is isomorphic to a quotient of $\xi(\A_{(\frp)})e_\pi$.


\begin{theorem}\label{main result} Fix data $L/K, S, \Pi$ and $\pi$ as above and set $r := {\rm rk}_\A(Y_\pi)$. Then there exists a canonical, locally-cyclic, full $\xi(\A)e_\pi$-submodule $\W^{\Pi,\pi}_{L/K,S}$ of 
$e_\pi\bigl(\calF\cdot{{{\bigcap}}}_{\A}^r(^{\Pi}\co^\times_{L,S})\bigr)$.\end{theorem}

\begin{definition}{\em We refer to elements of the module $\W^{\Pi,\pi}_{L/K,S}$ constructed in the above result as `Weil-Stark elements' for the data $L/K, S,  \Pi$ and $\pi$.} \end{definition}


\begin{remark}\label{cases of r}{\em  Write $r_{L,S}^\Pi$ for the maximal possible rank of a locally-free $\mathcal{A}$-module quotient of ${_{\Pi}}X_{L,S}$. Then it is clear that $e_\pi\neq 0 \Longleftrightarrow {\rm rk}_\A(Y_\pi)=r_{L,S}^\Pi$ and the construction will also imply  $\W^{\Pi,\pi}_{L/K,S}\not= (0) \Longleftrightarrow {\rm rk}_\A(Y_\pi)=r_{L,S}^\Pi$. In addition, in natural cases it is possible to describe $r_{L,S}^\Pi$ explicitly. 
For example, if $(_{\Pi}X_{L,S})_{\rm tf}$ is a locally-free $\mathcal{A}$-module, as is automatically true if $\mathcal{A}$ is a Dedekind domain, then 
 $r_{L,S}^\Pi = {\rm rk}_{\mathcal{A}}((_{\Pi}X_{L,S})_{\rm tf})$. In general, if one fixes any place $v$ in $S$, then $_{\Pi}Y_{L,S\setminus\{v\}}$ is a quotient of $_{\Pi}X_{L,S}$ and so $|S|\geq r_{L,S}^\Pi\ge | S^{\Pi}_v|$ with
\[ S^{\Pi}_v := \{ v' \in S\setminus\{v\}: \text{ the }\mathcal{A}\text{-module }\,H_0(G_{v'},\Pi)_{\rm tf} \,\text{ is both non-zero and locally-free}\},\]
where $G_{v'}$ denotes the decomposition subgroup in $G$ of some choice of place of $L$ above $v'$. Note also that, in the setting of Example \ref{exam1}(i), the $\mathcal{A}$-module $H_0(G_{v'},\Pi_\kappa)_{\rm tf}$ is non-zero and locally-free if $\kappa$ sends each element of $G_{v'}$ to the identity of $A$. 

}\end{remark}

\subsubsection{}\label{weil recall section} 

Under the present hypothesis on $S$, the  complex $R\Gamma_{c}((\co_{L,S})_\W,\ZZ)$ constructed in \cite[\S2.2]{bks} in terms of `Weil-\'etale cohomology' is an object of $D^{{\rm lf}}(\ZZ[G])$ that is defined up to canonical isomorphism. The linear dual complex 
\[ C_{L,S}:=R\Hom_\ZZ(R\Gamma_{c}((\co_{L,S})_\W,\ZZ),\ZZ)[-2]\]
is therefore also an object of $D^{{\rm lf}}(\ZZ[G])$ that is defined up to canonical isomorphism. 

In the sequel, for any finite set of places $T$ of $K$ that is disjoint from $S$, we  write $\mathcal{S}^T_S(L)^{{\rm tr}}$ for the $S$-relative $T$-trivialized transpose Selmer group of $\mathbb{G}_m$, as defined in \cite[\S2]{bks}. 

\begin{lemma}\label{Picoh} Let ${_\Pi}C$ denote the complex $\Pi\otimes_{\ZZ[G]}^{\mathbb{L}}C_{L,S}$ in $D(\A)$. Then ${_\Pi}C$ belongs to $D^{\rm lf}(\A)$ and is acyclic outside degrees zero and one with $H^0({_\Pi}C) = {^\Pi}\mathcal{O}_{L,S}^\times$ and $H^1({_\Pi}C)= {_\Pi}\mathcal{S}^\emptyset_S(L)^{{\rm tr}}$. \end{lemma}

\begin{proof} Since $C_{L,S}$ belongs to $D^{{\rm lf}}(\ZZ[G])$, and ${_\Pi}M$ is a locally-free $\A$-module for any locally-free $\ZZ[G]$-module $M$, it is clear that ${_\Pi}C$ belongs to $D^{\rm lf}(\A)$.

We next recall from \cite[Rem. 2.7]{bks} that the complex $C:= C_{L,S}$ is acyclic outside degrees zero and one and that there are canonical identifications $H^0(C) = \mathcal{O}_{L,S}^\times$ and $H^1(C)= \mathcal{S}^\emptyset_S(L)^{{\rm tr}}$. Since the complex ${_\Pi}C$ is obtained by applying the left derived functor $\Pi\otimes_{\ZZ[G]}^{\mathbb{L}}-$ to $C$, these facts imply that ${_\Pi}C$ is acyclic in degree greater than one and that there is a canonical identification $H^1({_\Pi}C)={_\Pi}H^1(C)={_\Pi}\mathcal{S}^\emptyset_S(L)^{{\rm tr}}$.

In addition, by taking a (bounded) resolution of the perfect complex $C$ by finitely generated projective $G$-modules, and using Remark \ref{ct remark}, one deduces the existence of a convergent cohomological spectral sequence of the form $H^b(G,\Pi\otimes_\ZZ H^{a}(C))\Longrightarrow\,H^{b+a}({_{\Pi}}C).$ 
 This spectral sequence implies that ${_\Pi}C$ is acyclic in all negative degrees and induces a canonical identification $H^0({_\Pi}C)={^{\Pi}}H^0(C)= {^\Pi}\mathcal{O}_{L,S}^\times$.
\end{proof}

\begin{remark}\label{lich rem}{\em The terminology of `Weil-\'etale cohomology with compact support' used above is motivated by the following facts (taken from \cite[Rem. 2.5]{bks}).

(i) Assume $L$ is a function field. Write $C_L$ for the corresponding smooth projective curve, $C_{L,W{\rm\acute e t}}$ for the Weil-\'etale site on $C_L$ that is defined by Lichtenbaum in \cite[\S2]{lichtenbaum} and $j$ for the open immersion ${\rm Spec}(\mathcal{O}_{L,S}) \longrightarrow C_L$. Then $R\Gamma_{c}((\co_{L,S})_\W,\ZZ)$ is canonically isomorphic to the complex $R\Gamma(C_{L,W{\rm\acute e t}},j_{!}\ZZ)$ that computes the cohomology with compact support of the constant sheaf $\ZZ$ on $C_{L,W{\rm\acute e t}}$.

(ii) Assume that $L$ is a number field. In this case there has as yet been no construction of a `Weil-\'etale topos' for $Y_S:= {\rm Spec}(\mathcal{O}_{L,S})$ with all of the properties that are conjectured by Lichtenbaum in \cite{lichtenbaum1}. However, if $\overline{Y_S}$ is a compactification of $Y_S$ and $\phi$ is the natural inclusion $Y_S\subset \overline{Y_S}$, then it can be shown that, should such a topos exist with all of the expected properties, the groups $H^i(R\Gamma_{c}((\co_{L,S})_\W,\ZZ))$ would be canonically isomorphic to the groups $H^i_c(Y_S,\ZZ) := H^i(\overline{Y_S},\phi_!\ZZ)$ discussed in \cite{lichtenbaum1}.}
\end{remark}

\subsubsection{}\label{epi} Turning now to the proof of Theorem \ref{main result}, we use reduced determinant functors (cf. \S\ref{332}). For each $\frp\in{\rm Spec}(\Lambda)$ we fix, as we may, an $\A_{(\frp)}$-basis $\underline{b}_{\pi,\frp}$ 
of $Y_{\pi,(\frp)}$. We then note that this choice of basis induces a composite homomorphism of $\xi(\A_{(\frp)})$-modules
\begin{align}\label{Thetap}\Theta_{\underline{b}_{\pi,\frp}}:\D_{A}(\calF\cdot {_\Pi}C)
\cong&\,\,\D^\diamond_{A}\bigl(\calF\cdot{^\Pi}\mathcal{O}_{L,S}^\times\bigr)\otimes_{\zeta(A)}\D^\diamond_{A}\bigl(\calF\cdot{_\Pi}\mathcal{S}^\emptyset_S(L)^{\rm tr}\bigr)^{-1}\\ \notag
\to&\,\,\D^\diamond_{Ae_\pi}\bigl(e_\pi(\calF\cdot{^\Pi}\mathcal{O}_{L,S}^\times)\bigr)\otimes_{\zeta(A)e_\pi}\D^\diamond_{Ae_\pi}\bigl(e_\pi(\calF\cdot{_\Pi}X_{L,S})\bigr)^{-1}\\ \notag
\cong&\,\,\D^\diamond_{Ae_\pi}\bigl(e_\pi(\calF\cdot{^\Pi}\mathcal{O}_{L,S}^\times)\bigr)\otimes_{\zeta(A)e_\pi}\D^\diamond_{Ae_\pi}\bigl(e_\pi(\calF\cdot Y_\pi)\bigr)^{-1}\\ \notag
\cong&\,\,\D^\diamond_{Ae_\pi}\bigl(e_\pi(\calF\cdot{^\Pi}\mathcal{O}_{L,S}^\times)\bigr)\\ \notag
=&\,\,e_\pi\bigl(\calF\cdot{{{\bigcap}}}^r_{\A}{^\Pi}\mathcal{O}_{L,S}^\times\bigr).
\end{align}
Here the first isomorphism is induced by the `passage to cohomology' map (\ref{ptc iso}) with $X = {_\Pi}C$ and the explicit descriptions of cohomology given in Lemma \ref{Picoh}, the arrow is induced by multiplication by $e_\pi$ and the fact that $\calF\cdot\mathcal{S}^\emptyset_S(L)^{\rm tr} = \calF\cdot X_{L,S}$, the second isomorphism is induced by $\pi$, and the third isomorphism is induced by the isomorphism

\begin{equation}\label{diamond}\zeta(A)e_\pi\cong\D^\diamond_{Ae_\pi}\bigl(e_\pi(\calF\cdot Y_{\pi})\bigr), \quad e_\pi\mapsto e_\pi\wedge_{i=1}^{i=r}b_{\frp,i}.\end{equation}
The final equality in (\ref{Thetap}) and the fact that (\ref{diamond}) is an isomorphism follow from \cite[Prop. 5.9]{bses}.

If one regards $\D_{\A_{(\frp)}}({_\Pi}C_{(\frp)})$ as a submodule of $\D_{A}(\calF\cdot {_\Pi}C)=\calF\cdot\D_{\A_{(\frp)}}({_\Pi}C_{(\frp)})$ in the natural way (cf. \S\ref{332}), then \cite[Lem. 4.13]{bses} implies that the $\xi(\mathcal{A}_{(\frp)})e_\pi$-lattice $$\W_\frp:=\Theta_{\underline{b}_{\pi,\frp}}\bigl(\D_{\A_{(\frp)}}({_\Pi}C_{(\frp)})\bigr)$$ is independent of the choice of basis $\underline{b}_{\pi,\frp}$ of $Y_{\pi,(\frp)}$.  We therefore obtain a canonical $\xi(\A)e_\pi$-submodule of $e_\pi(\calF\cdot{{{\bigcap}}}^r_{\A}{^\Pi}\mathcal{O}_{L,S}^\times)$ by setting
$$\W_{L/K,S}^{\Pi,\pi}:={\bigcap}_{\frp\in {\rm Spec}(\Lambda)}\W_{\frp}.$$ 
In addition, as the $\xi(\A_{(\frp)})$-module $\D_{\A_{(\frp)}}({_\Pi}C_{(\frp)})$ is free of rank one,  $\W_\frp$ is isomorphic to a quotient of $\xi(\A_{(\frp)})e_\pi$. The proof of Theorem \ref{main result} is thus completed by the following result.

\begin{lemma}\label{CRlemma} The module $\W^{\Pi,\pi}_{L/K,S}$ is full in
$e_\pi\bigl(\calF\cdot{{{\bigcap}}}_{\A}^r(^{\Pi}\co^\times_{L,S})\bigr)$ and, for each $\frp\in {\rm Spec}(\Lambda)$, satisfies $(\W_{L/K,S}^{\Pi,\pi})_{(\frp)}=\W_\frp$.
\end{lemma}
\begin{proof}  It will be enough to prove the existence of a finite set $U$ of primes of $\Lambda$ with the following property: for every prime $\frp\notin U$ there exists an $\A_{(\frp)}$-basis $\underline{b}_{\pi,\frp}$ of $Y_{\pi,(\frp)}$, for which the $\xi(\A)e_\pi$-module $\Theta_{\underline{b}_{\pi,\frp}}(\D_{\A}({_\Pi}C))$ is independent of the choice of $\frp\notin U$. Indeed, since for each prime $\frp$ of $\Lambda$ one has $\D_{\A_{(\frp)}}({_\Pi}C_{(\frp)})=\D_{\A}({_\Pi}C)_{(\frp)}$, we will then be able to apply the general result \cite[Prop. (4.21)(vii)]{curtisr} to the $\xi(\A)e_\pi$-module $M=\Theta_{\underline{b}_{\pi,\frp_0}}(\D_{\A}({_\Pi}C))$, defined via an arbitrary prime $\frp_0\notin U$. We also note that  $M$ is a full lattice in $e_\pi\bigl(\calF\cdot{{{\bigcap}}}_{\A}^r(^{\Pi}\co^\times_{L,S})\bigr)$ since $\calF\cdot\D_{\A}({_\Pi}C)=\D_A(\calF\cdot{_\Pi}C)$ by the commutativity of (\ref{eofs}).

The required independence condition in turn is valid because Roiter's Lemma \cite[Lem. (31.6)]{curtisr} implies that there exists a free $\A$-submodule $X$ of $Y_\pi$ of rank $r$, and hence also a finite set $U$ of primes of $\Lambda$ such that $X_{(\frp)}=Y_{\pi,(\frp)}$ for all $\frp\notin U$. Fixing an arbitrary $\A$-basis $\underline{x}=\{x_i\}_{i \in [r]}$ of $X$ we write $\underline{x}_\frp$ for the induced basis of $Y_{\pi,(\frp)}$ for each $\frp\notin U$.

Then $\Theta_{\underline{x}_{\frp}}$ is independent of $\frp\notin U$ as it is induced by the isomorphism of $\zeta(A)e_\pi$-modules
$$\zeta(A)e_\pi\cong\D^\diamond_{Ae_\pi}\bigl(e_\pi(\calF\cdot X)\bigr)=\D^\diamond_{Ae_\pi}\bigl(e_\pi(\calF\cdot Y_{\pi})\bigr), \quad e_\pi\mapsto e_\pi\wedge_{i=1}^{i=r}x_{i}$$
as in (\ref{diamond}) (and had a priori no other dependence on the choice of basis).
\end{proof}



\begin{remark}\label{unramified extensions}{\em By a classical result of Fr\"ohlich \cite{af}, there exist unramified Galois extensions of number fields $L/K$ with $G$ isomorphic to any prescribed finite group. In any such case one can take the set $S$ in \S\ref{dataasin} to be $S_K^\infty$. Then all  places in $S$ split completely in $L$ and Lemma \ref{fe}(i) implies $a := |S_K^\infty|$
is the maximal integer for which ${\bigcap}^a_{\ZZ[G]}\mathcal{O}_{L}^\times\not= (0)$. However, the above argument does not construct Weil-Stark elements in the latter module since $X_L$ has no $\ZZ[G]$-locally-free quotient of rank $a$ (cf. Remark \ref{cases of r}). To overcome this problem, we write $I_G$ for the augmentation ideal of $\ZZ[G]$, $\eins$ for the trivial character of $G$, $\tilde e$ for the idempotent $1-e_{\eins}$ of $\zeta(\QQ[G])$ and $B$ for the algebra $\QQ[G]\tilde e = \QQ\cdot I_G$. We write $\B$ for the order $\ZZ[G]\tilde e$ in $B$  and $\Pi$ for $\B$, considered as a $(\B,\ZZ[G])$-bimodule in the obvious way. Then $\B$ is locally-Gorenstein only when $G$ is cyclic. In all cases, however, the $G$-module $X_L$ is isomorphic to the direct sum $I_G\oplus\ZZ[G]^{\oplus (a-1)}$ and so the $\B$-module $Y:=({_{\Pi}}X_{L})_{\rm tf}$ is free of rank $a$ and one can take the map $\pi$ in (\ref{choice of pi}) to be the projection ${_{\Pi}}X_{L}\to Y$. In this case one has $e_\pi=1$ 
and so the construction of Theorem \ref{main result} gives a canonical $\xi(\B)$-module 
\[\W^{a}_{L/K}:=\W^{\Pi,\pi}_{L/K,S^\infty_K} \subset \QQ\cdot{{\bigcap}}^a_{\B}\bigl({^{\Pi}}\mathcal{O}_{L}^\times\bigr)= \QQ\cdot{{\bigcap}}^a_{\ZZ[G]}\mathcal{O}_{L}^\times,\]
where the last equality is a consequence of Lemma \ref{fe}.} 
\end{remark}

\subsection{Cyclotomic units} 

As  first concrete example, in this section we describe the link between cyclotomic units and Weil-Stark elements. To do this we take $K= \QQ$ and fix a non-trivial real abelian extension $L$ of $K$ in $\CC$ and set $G := \Gal(L/K)$. We write $m= m_L$ for the conductor of $L$ so that $L\subseteq \QQ(e^{2\pi i/m})$ (by the Kronecker-Weber Theorem). 

We note that, since $L$ is real, the $G$-module $Y_{L,\{\infty\}}$ is a free of rank one. We write $S$ for the set of places of $\QQ$ comprising $\infty$ and all prime divisors of $m$ and consider the canonical surjective homomorphism of $G$-modules $\pi :  X_{L,S} \to Y_\pi:=Y_{L,\{\infty\}}$.
%


\begin{theorem}\label{cyclo-weil thm} Fix $L/K, S$ and $\pi$ as above and regard $\Pi := \ZZ[G]$ as a $(\ZZ[G],\ZZ[G])$-bimodule in the obvious way. Then, in $\QQ\otimes_\ZZ \mathcal{O}^\times_{L,S}$, there is an equality 
\[ \W^{\Pi,\pi}_{L/K,S} = \ZZ[G]\cdot {\rm N}_{\QQ(e^{2\pi i/m})/L}(1-e^{2\pi i/m})^{1/2}.\] 
\end{theorem}

\begin{proof} We consider the composite isomorphism of graded invertible $\RR[G]$-modules
\[ \vartheta_L: \RR \cdot {\rm d}_{\ZZ[G]}(C_{L,S}) \xrightarrow{\sim} 
{\rm d}_{\RR[G]}
(\RR \cdot \mathcal{O}_{L,S}^\times) \otimes_{\RR[G]} {\rm d}_{\RR[G]}(\RR \cdot X_{L,S})^{-1} \xrightarrow{\sim} (\RR[G],0).\]
%
Here the first arrow is the canonical passage to cohomology map and the second is induced by sending each element $u\otimes \theta$ in the tensor product to $\theta({\rm d}_{\RR[G]}(R_{L,S})(u))$.

%
%

The equivariant Tamagawa number conjecture for the pair $(h^0({\rm Spec} (L), \ZZ[G])$ is valid by work of Greither and the first author \cite{bg} and of Flach \cite{fg}. This fact combines with \cite[Prop. 3.4]{bks} to imply the $\ZZ[G]$-module ${\rm d}_{\ZZ[G]}(C_{L,S})$ is free of rank one with a 
basis  $\mathfrak{z}_{L}$ such that 
\begin{equation}\label{cyclo etnc} \vartheta_L(\mathfrak{z}_L) = \bigl({\sum}_{\chi\in \widehat{G}} L^\ast_{S}(\chi^{-1}, 0)e_\chi,0\bigr),\end{equation}
where $L^\ast_{S}(\chi^{-1}, 0)$ denotes the leading term at $s=0$ of the $S$-truncated $L$-series $L_{S}(\chi^{-1}, s)$.

In the rest of the argument we shall show that this fact implies the claimed result. To do this we note that the $G$-module $Y_\pi = Y_{L,\{\infty\}}$ is free of rank one, with basis the place $w$ that is induced by the given inclusion $L \subset \CC$. By using this basis element,  the constructions in (\ref{Thetap}) and (\ref{diamond}) (with $A = \QQ[G]$, $\Pi = \ZZ[G]$, $\pi$ as above and $r = 1$) therefore give rise to a homomorphism of $\QQ[G]$-modules $\Theta_w: \QQ \cdot {\rm d}_{\ZZ[G]}(C_{L,S}) \to e_\pi\bigl(\QQ\cdot\mathcal{O}_{L,S}^\times\bigr)$.

Now the module $\W^{\Pi,\pi}_{L/K,S}$ can be computed by using, for every prime $p$, the map $\Theta_{b_{\pi,p}}$ that arises from the basis $b_{\pi,p}$ of the $\ZZ_{(p)}[G]$-module $Y_{\pi,(p)}$ that corresponds to $w$. In particular, since  $\mathfrak{z}_L$ is a basis of the $G$-module ${\rm d}_{\ZZ[G]}(C_{L,S})$, one computes in this way that $\W^{\Pi,\pi}_{L/K,S}$ is equal to the $G$-module generated by $\Theta_{w}(\mathfrak{z}_L)$. To deduce the claimed result it is therefore enough to show that the equality (\ref{cyclo etnc}) implies 
 $\Theta_{w}(\mathfrak{z}_L)$ is equal to the element $\varepsilon_L := {\rm N}_{\QQ(e^{2\pi i/m})/L}(1-e^{2\pi i/m})^{-1/2}$. 
 
The key point in verifying this is that for each $\chi$ in $\widehat{G}$ the first derivative $L_{S}'(\chi, s)$ of $L_{S}(\chi, s)$ is holomorphic at $s=0$ and such that
 \[ L_{S}'(\chi, 0) = -(1/2){\sum}_{\sigma \in \Gal(\QQ(e^{2\pi i/m})/\QQ)} \chi(\sigma)\log|(1- \sigma(e^{2\pi i/m}))| = {\sum}_{g \in G} \chi(g)\log| g(\varepsilon_L)|,
    \]
where the first equality is proved, for example, in \cite[Chap. 3, \S5]{tate} and the second follows directly from the definition of $\varepsilon_L$. This equality in turn implies that 
\begin{equation}\label{reg interpretation} R_{L,S}(\varepsilon_L) = \bigl({\sum}_{\chi\in \widehat{G}} L'_{S}(\chi^{-1}, 0)e_\chi\bigr)\cdot (w-w_0),\end{equation}
where $w_0$ is any choice of place of $L$ that lies above a prime divisor of $m$. 

To proceed we note next that, by a direct comparison of the definitions of all involved maps, the $e_\pi$ component of the isomorphism $\vartheta_L$ is equal to the composite 
\[ \bigl(\RR \cdot {\rm d}_{\ZZ[G]}(C_{L,S})\bigr) \xrightarrow{\RR\otimes_\QQ\Theta_w} e_\pi\bigl(\RR\cdot\mathcal{O}_{L,S}^\times\bigr) \to \RR[G]e_\pi\]
in which the second map sends each element $u$ to the unique element $c_u$ of $\RR[G]e_\pi$ that is specified by the equality $R_{L,S}(u) = c_u\cdot (w-w_0)$. One therefore has 
\begin{align*} R_{L,S}(\Theta_w(\mathfrak{z}_L)) =&\, e_\pi\vartheta_L(\mathfrak{z}_L)\cdot (w-w_0)\\
                                                 =&\, \bigl({\sum}_{\chi\in \widehat{G}} L^\ast_{S}(\chi^{-1}, 0)e_\chi e_\pi)\cdot (w-w_0)\\
                                                 =&\, \bigl({\sum}_{\chi\in \widehat{G}} L'_{S}(\chi^{-1}, 0)e_\chi\bigr)\cdot (w-w_0)\\
                                                 =&\, R_{L,S}(\varepsilon_L),\end{align*}
Here the second displayed equality follows from 
(\ref{cyclo etnc}), the third from the fact that our choice of $\pi$ combines with the explicit formula for the order of vanishing at $s=0$ of $L_{S}(\chi^{-1},s)$ obtained in \cite[Chap. I, Prop. 3.4]{tate} to 
imply that $L^\ast_{S}(\chi^{-1}, 0)e_\chi e_\pi = L'_{S}(\chi^{-1}, 0)e_\chi$ for all $\chi$ in $\widehat{G}$, and the fourth directly from (\ref{reg interpretation}). In particular, since the map $R_{L,S}$ is injective, this implies the required equality  
$\Theta_w(\mathfrak{z}_L) = \varepsilon_L$. \end{proof} 

\section{Weil-Stark elements over locally-Gorenstein orders} \label{sgp}

In this section we prove that if $\A$ is locally-Gorenstein, then the elements constructed in Theorem \ref{main result} encode detailed information about the structure of Selmer groups. 

We assume that all reduced exterior products are normalised as in \S\ref{We}.  

\subsection{Statement of the main result}\label{Selmermodules} We fix a finite set $T$ of places of $K$ with $S\cap T=\emptyset$ and such that the group $L_T^\times := \{ a\in L^\times : {\rm ord}_w(a-1)>0 \text{ for all } w\in T_L \}$ is torsion-free. We set $\co_{L,S,T}^\times := \mathcal{O}_{L,S}^\times\cap L_T^\times$ and write ${\rm Cl}_S^T(L)$ for the quotient of the group of fractional $\mathcal{O}_L$-ideals
whose support is disjoint from $(S\cup T)_L$ by the subgroup
of principal ideals with a generator congruent to $1$ modulo all
places in $T_{L}$. We recall that ${\rm Cl}_S^T(L)$ is a $G$-module extension of ${\rm Cl}_S(L)$ that also lies in a canonical exact sequence
\begin{equation}\label{selmer lemma II} 0 \longrightarrow {\rm Cl}_{S}^T(L) \longrightarrow
\mathcal{S}_S^T(L)^{\rm tr}
\longrightarrow X_{L,S} \longrightarrow 0,\end{equation} where again $\mathcal{S}_S^T(L)^{\rm tr}$ denotes the transpose Selmer group of $\mathbb{G}_m$ (cf. \cite[\S2]{bks}).
We also set  
\begin{equation*}\label{delta T def} \gamma_{T,\Pi} := \iota_\Pi\bigl({\prod}_{v \in T}{\rm Nrd}_{\QQ[G]}(1-\sigma_{v}^{-1}{\rm N}v)\bigr)\in \zeta(A)^\times,\end{equation*}
where $\sigma_{v}$ is the Frobenius automorphism of $w_v$ in $G$ and ${\rm N}v$ the absolute norm of $v$, and the containment is a  consequence of the commutative diagram (\ref{key commute}). 

We next define a `projector' associated to the bimodule $\Pi$ by setting
\[ {\rm pr}_\Pi := {\sum}_{g \in G}\chi_\Pi(g)\otimes g^{-1}\in A[G],\]
where $\chi_\Pi$ is the $A$-valued character of (the free $A$-module) $\QQ\otimes_\ZZ \Pi$. For an additive map $\epsilon: A\to \QQ$ we write $\epsilon_G$ for the map $A[G] \to \QQ[G]$ sending each ${\sum}_{g \in G}a_gg$ to ${\sum}_{g \in G}\epsilon(a_g)g$.

\begin{example}\label{nobimodulechar}{\em In the setting of Example \ref{exam1}(i), take $\kappa$ to be the identity map $\Lambda[G]\to\Lambda[G]$ and set $\Pi:=\Pi_\kappa=\Lambda[G]$. 
Then $\chi_\Pi(g) = g$ for all $g \in G$ and, for the map $\epsilon:\QQ[G] \to \QQ$ given by projection onto the coefficient of the identity element of $G$, one can check that $\epsilon_G(x\cdot {\rm pr}_\Pi) = x$ for all $x\in \QQ[G]$. 
%
}\end{example}

In the following result we use the $\zeta(\A)$-ideal $\delta(\A)$ (cf. \S\ref{Whitehead section}) and the notions of `locally-quadratic presentation' and `non-commutative Fitting invariant' (cf. \S\ref{fi review}). In addition, for a $\xi(\mathcal{A})$-module $X$, in claim (ii) we write $X^{e_\pi}$ for the submodule $\{x \in X: e_\pi \cdot x = x \text{ in } \mathcal{F}\otimes_\Lambda X\}$.

\begin{theorem}\label{main result2} Fix data as in Theorem \ref{main result} and an auxiliary set $T$ as above. Assume $\A$ is locally-Gorenstein and the bimodule $\Pi$ satisfies Hypothesis \ref{Pi hyp}. Set $\W_T:=\gamma_{T,\Pi}\cdot(\W_{L/K,S}^{\Pi,\pi})$. 
\begin{itemize}
\item[(i)] $\W_T$ is contained in ${{{\bigcap}}}^{r}_{\A}{^{\Pi}}\co^\times_{L,S,T}$.
\item[(ii)] There exists a locally-quadratic presentation $h$ of the $\mathcal{A}$-module ${_\Pi}\mathcal{S}_S^T(L)^{\rm tr}$ such that  
\begin{equation*}\label{PiFit}\{(\wedge_{i=1}^{i=r}\varphi_i)(\varepsilon)\!:\!\, \varphi_i \in \Hom_{\mathcal{A}}(^{\Pi}\mathcal{O}_{L,S,T}^{\times},\mathcal{A}),\,\varepsilon\in\W_T\} = {\rm Fit}_{\mathcal{A}}^r\bigl(h\bigr).\end{equation*}
If $\xi(\mathcal{A})$ is locally-Gorenstein, there exists a perfect bilinear pairing of $\xi(\mathcal{A})$-modules 
\[ \bigl(({{\bigcap}}^{r}_{\A}{^{\Pi}}\co^\times_{L,S,T})^{e_\pi}/\mathcal{W}_T\bigr) \times  
\bigl(\xi(\mathcal{A})^{e_\pi}/{\rm Fit}_{\mathcal{A}}^r(h)\bigr)\to \zeta(A)/\xi(\mathcal{A}).\]
\item[(iii)] Fix $S'\subseteq S$ with $S^\infty_K\subseteq S'$ and so the natural composite map ${_\Pi}X_{L,S'} \to {_\Pi}X_{L,S}  \xrightarrow{\pi} Y_\pi$ has finite cokernel. Then for every $a\in \delta(\A)\cdot{\rm Ann}_{\A}({\rm Tor}^{G}_1(\check\Pi,\Hom_\ZZ(\mathcal{O}_{L,S',T}^\times,\ZZ)))$, $(\varphi_i)_{1\leq i\leq r}\in \Hom_{\mathcal{A}}(^{\Pi}\mathcal{O}_{L,S,T}^\times,\mathcal{A})^{r}$, $\varepsilon\in \W_T$ and additive map $\epsilon:\mathcal{A} \to \ZZ$, the element 
 \[ \epsilon_G\bigl(\iota_\A\bigl(a\cdot(\wedge_{i=1}^{i=r}\varphi_i)(\varepsilon)\bigr)\cdot {\rm pr}_{\check\Pi}\bigr)\]
 belongs to $\ZZ[G]$ and annihilates ${\rm Cl}^T_{S'}(L)$.
\end{itemize}
\end{theorem}

\begin{remark}{\em 
The proof of Theorem \ref{main result2}(i) will also show that $\W^{\Pi,\pi}_{L/K,S}\subseteq ({{{\bigcap}}}^{r}_{\A}{^{\Pi}}\co^\times_{L,S})_{(\frp)}$ for all $\p \in {\rm Spec}(\Lambda)$ for which  $(^{{\Pi}}\co^\times_{L,S})_{(\frp)}$ is torsion-free.}\end{remark}

%
%

%

\begin{remark}\label{refine rubin}{\em In the setting of Example \ref{nobimodulechar}, 
Theorem \ref{main result2}(iii) implies that any element of $\delta(\ZZ[G])\cdot\{(\wedge_{i=1}^{i=r}\varphi_i)(\varepsilon):\varphi_i\in\Hom_G(\co_{L,S,T}^\times,\ZZ[G]),\varepsilon\in\W_T\}$ belongs to $\ZZ[G]$ and annihilates ${\rm Cl}^T_{S'}(L)$. Here we have set $S':=S_K^\infty\cup V\cup\{v_0\}$, with $V$ the subset of $S$ comprising places that split completely in $L/K$ and $v_0$ any place in $S\setminus V$. In particular, if $K=\QQ$ and $L$ is real abelian, this fact combines with Theorem \ref{cyclo-weil thm} to recover the main result of Rubin \cite{rubininv}. 
}\end{remark} 

\begin{remark}\label{unr rem}{\em If $\A$ is not locally-Gorenstein, then our methods still determine properties of Weil-Stark elements (albeit less precisely than Theorem \ref{main result2}). Understanding such properties is of interest since, given $\Pi$, one can sometimes increase the maximal permissible rank $r^\Pi_{L,S}$ in Remark \ref{cases of r}, and thereby extend the range of Theorem \ref{main result}, by considering bimodules associated to ring extensions of $\A$ that are not locally-Gorenstein (as, for example, in Remark \ref{unramified extensions}). For an example of such properties see Remark \ref{A2remark} below. }\end{remark}

\subsection{Reduction steps}

\subsubsection{}\label{cohdescription} We first discuss a useful variant of the complex $C_{L,S}$ and of the module of Weil-Stark elements. 
To do this we note that each place $v$ in our fixed set $T$ is unramified in $L$, and hence that there exists an exact sequence of $G$-modules of the form 
\begin{equation}\label{resolve kappa} 0 \to \ZZ[G] \xrightarrow{x\to x(1-\sigma_v^{-1}{\rm N}v)} \ZZ[G] \to {\bigoplus}_{w \in \{v\}_L}\kappa^\times_w\to 0,\end{equation}
where $\kappa_w$ denotes the residue field of $w$. This resolution implies that the complex $\bigl({\bigoplus}_{w \mid v}\kappa^\times_w\bigr)[0]$ defines an object $D^{\rm lf}(\ZZ[G])$, so by Lemma \ref{Picoh} we can fix an exact triangle in $D^{\rm lf}(\ZZ[G])$ 
\begin{equation}\label{Ttriangle}C_{L,S,T}\longrightarrow C_{L,S}\longrightarrow\bigl({\bigoplus}_{w\in T_L}\kappa_w^\times\bigr)[0]\longrightarrow,\end{equation}
in which the second arrow is the canonical morphism constructed in \cite[\S2.2]{bks}. The complex $C_{L,S,T}$ is then defined up to canonical isomorphism (in $D^{\rm lf}(\ZZ[G])$), is acyclic outside degrees zero and one, and is such that there are identifications $H^0(C_{L,S,T}) = \mathcal{O}_{L,S,T}^\times$ and $H^1(C_{L,S,T})= \mathcal{S}_S^T(L)^{{\rm tr}}$. 
In particular, the complex 
 \[ {_\Pi}C_T:=\Pi\otimes_{\ZZ[G]}^{\mathbb{L}}C_{L,S,T}\]
 lies in an exact triangle 
\begin{equation*}\label{Ttriangle2} {_\Pi}C_T \longrightarrow {_\Pi}\,C_{L,S} \longrightarrow {_\Pi}\bigl({\bigoplus}_{w\in T_L}\kappa_w^\times\bigr)[0]\longrightarrow,\end{equation*}
in $D^{\rm lf}(\A)$ induced by (\ref{Ttriangle}) and the same argument as in Lemma \ref{Picoh} implies ${_\Pi}C_T$ is acyclic outside degrees zero and one and gives identifications $H^0({_\Pi}C_T) = {^\Pi}\mathcal{O}_{L,S,T}^\times$ and $H^1({_\Pi}C_T)= {_\Pi}\mathcal{S}_S^T(L)^{{\rm tr}}$. By applying the functor ${\rm d}_\A(-)$ to the latter exact triangle one finds that 
\begin{equation}\label{dT} \D_{\A}({_\Pi}C_T) = \D_\A\bigl({_\Pi}\bigl({\bigoplus}_{w\in T_L}\kappa_w^\times\bigr)[-1]\bigr)\cdot \D_\A( {_\Pi}C_{L,S}) = \gamma_{T,\Pi}\cdot \D_\A({_\Pi}C_{L,S}),\end{equation}
where the second equality follows from the resolutions (\ref{resolve kappa}) and the commutativity of (\ref{key commute}). 

For $\frp\in {\rm Spec}(\Lambda)$ we fix an $\A_{(\frp)}$-basis $\underline{b}_{\pi,\frp}$ of $Y_{\pi,(\frp)}$. Then, setting 
\begin{equation}\label{WT} \W_{T,\frp}:=\Theta_{\underline{b}_{\pi,\frp}}\bigl(\D_{\A_{(\frp)}}\bigl(({_\Pi}C_{T})_{(\frp)}\bigr)\bigr)\quad\text{and}\quad \W_T={\bigcap}_{\frp \in {\rm Spec}(\Lambda)}\W_{T,\frp},\end{equation}
the equality (\ref{dT}) implies that $\W_T= \gamma_{T,\Pi}\cdot\bigl(\W_{L/K,S}^{\Pi,\pi}\bigr)$ (so the notation $\W_T$ is consistent with that of Theorem \ref{main result2}). From Theorem \ref{main result} and Lemma \ref{CRlemma}, we can therefore deduce that $\W_T$ is a locally-cyclic $\xi(\A)$-module and that $(\W_T)_{(\frp)}=\W_{T,\frp}$ for all $\frp\in {\rm Spec}(\Lambda)$.

\subsubsection{}\label{334} 
%

%
 


The key to our proof of Theorem \ref{main result2} is to show ${_\Pi}C_T$ is represented by a complex that is well-suited to the computation of $\W_T$. For this, we use the surjective map of $\A$-modules 
\[  \pi_T:{_\Pi}\mathcal{S}_S^T(L)^{\rm tr}\to{_\Pi}X_{L,S}\xrightarrow{\pi}Y_\pi,\]
where the first arrow is induced by the exact sequence (\ref{selmer lemma II}).
%
We also recall from Remark \ref{cases of r} that the rank $r:={\rm rk}_{\A}(Y_\pi)$ satisfies $r \le |S|$. 

\begin{proposition}\label{yoneda} There exists an isomorphism $\vartheta: P^\bullet \to {_\Pi}C_T$ in $D^{\rm lf}(\A)$ with all of the following properties.
\begin{itemize}\item[(i)] $P^\bullet$ has the form $P^0\xrightarrow{\phi} P$, where $P^0$, resp. $P$, is a locally-free $\A$-module (placed in degree zero), resp. a free $\A$-module, of finite rank. In addition, $d:={\rm rk}_{\A}(P)={\rm rk}_{\A}(P^0)> |S|.$ 
\item[(ii)] The isomorphism $\vartheta$ induces an exact sequence of $\A$-modules
$$0\to{^\Pi}\co^\times_{L,S,T}\stackrel{\iota}{\longrightarrow}P^0\stackrel{\phi}{\longrightarrow}P\stackrel{\varpi}{\longrightarrow}{_\Pi}\mathcal{S}_S^T(L)^{\rm tr}\to 0.$$
\item[(iii)] Fix $\frp\in {\rm Spec}(\Lambda)$ and a basis $\underline{b}_{\pi,\frp} = \{b_{\frp,i}\}_{i \in [r]}$ of the  $\mathcal{A}_{(\frp)}$-module $Y_{\pi,{(\frp)}}$.
Then there is 
an $\mathcal{A}_{(\frp)}$-basis $\hat{\underline{b}} = \{\hat b_{i}\}_{i \in [d]}$ of $P_{(\frp)}$ with
\begin{equation*}\label{basis choice} (\pi_T\circ\varpi)_{(\frp)}(\hat b_{i}) = \begin{cases} b_{\frp,i}, &\text{if $i\in [r]$}\\
 0, &\text{if $i \in [d]\setminus [r]$.}
 \end{cases}\end{equation*}
In particular, if $\{\hat b_i^{\ast}\}_{i \in [d]}$ is the dual basis of $\hat{\underline{b}}$ in $\Hom_{\A_{(\frp)}}(P_{(\frp)},\A_{(\frp)})$, then for each $i \in [r]$, one has $\hat b_i^{\ast}\circ \phi_{(\frp)} = 0$. 
\item[(iv)] Assume $\A$ is locally-Gorenstein. Fix $\frp\in {\rm Spec}(\Lambda)$, an isomorphism $\psi:P^0_{(\frp)}\cong P_{(\frp)}$ of $\mathcal{A}_{(\frp)}$-modules and an $\mathcal{A}_{(\frp)}$-basis $\hat{\underline{b}}$ of $P_{(\frp)}$ as in (iii). Then the element
\[\varepsilon_{\hat{\underline{b}}}:=  \left({\wedge}_{i=r+1}^{i=d}\calF\cdot(\hat b_i^{\ast}\circ\phi_{(\frp)}\circ\psi^{-1})\right)({\wedge}_{c=1}^{c=d}\hat b_{c})\in {{\bigwedge}}_{A}^r(\calF \cdot P)\]
belongs to ${{{\bigcap}}}^r_{\A_{(\frp)}}({^\Pi}\co^\times_{L,S,T})_{(\frp)}$ and generates the $\xi(\A_{(\frp)})$-module $\W_{T,\frp}$, both viewed as sublattices of ${{\bigwedge}}_{A}^r(\calF \cdot P)$ via $\bigwedge_A^r(\calF\cdot(\psi\circ\iota_{(\frp)}))$.
\end{itemize}
\end{proposition}
\begin{proof} A standard construction in homological algebra (as in \cite[Rapport, Lem. 4.7]{del}) implies the existence of an isomorphism $\tilde\vartheta: \tilde P^\bullet \to C_T$ in $D^{\rm lf}(\ZZ[G])$, with $\tilde P^\bullet$ of the form $\tilde P^0\xrightarrow{\tilde\phi}\tilde P$ where $\tilde P^0$ is a  finitely generated $\ZZ[G]$-module of projective dimension at most one and $\tilde P$ a free $\ZZ[G]$-module of finite rank $d>|S|$. 
The associated exact sequence 
\begin{equation}\label{beforebasechange}0\to\co^\times_{L,S,T}\stackrel{\tilde\iota}{\longrightarrow}\tilde P^0\stackrel{\tilde\phi}{\longrightarrow}\tilde P\stackrel{\tilde\varpi}{\longrightarrow}\mathcal{S}_S^T(L)^{\rm tr}\to 0.\end{equation}
then combines with the fact $\co^\times_{L,S,T}$ is $\ZZ$-torsion-free to imply $\tilde P^0$ is also $\ZZ$-torsion-free. Hence, since any $G$-module of projective dimension at most one is cohomologically-trivial, $\tilde P^0$ is a projective $\ZZ[G]$-projective (by \cite[Th. 8]{aw}) and thus, by Swan's Theorem \cite[Th. (32.11), Vol. I]{curtisr}, locally-free. In addition, the exactness of (\ref{selmer lemma II}) (and the argument of Lemma \ref{fe}(i)) implies the $\QQ[G]$-modules $\QQ\cdot\co^\times_{L,S,T}$ and $\QQ\cdot\mathcal{S}_S^T(L)^{\rm tr}$ are isomorphic. From the exactness of (\ref{beforebasechange}), it then follows that $d:={\rm rk}_{\ZZ[G]}(\tilde P)={\rm rk}_{\ZZ[G]}(\tilde P^0)$.


The base change $\vartheta:={_\Pi}\tilde\vartheta=\Pi\otimes_{\ZZ[G]}^{\mathbb{L}}\vartheta$ of $\tilde\vartheta$ gives an isomorphism from $P^\bullet:={_\Pi}\tilde P^\bullet$ to ${_\Pi}C_T$ in $D^{\rm lf}(\A)$ and, by Lemma \ref{trace sequences} applied to $\tilde\phi$, induces the exact sequence claimed in (ii) with 
$P^0:={_\Pi}\tilde P^0$, $P:={_\Pi}\tilde P$, $\phi:={_\Pi}\tilde\phi$, $\varpi:={_\Pi}\tilde\varpi$ and $\iota:=T_{\Pi,\tilde P^0}^{-1}\circ{^\Pi}\tilde\iota$. The choice of $\tilde P^\bullet$ also readily implies that $P^\bullet=[P^0\xrightarrow{\phi} P]$ satisfies claim (i).

To proceed we fix $\frp\in {\rm Spec}(\Lambda)$ and a basis $\underline{b}_{\pi,\frp} = \{b_{\frp,i}\}_{i\in [r]}$ of the $\mathcal{A}_{(\frp)}$-module $Y_{\pi,{(\frp)}}$. Since the composite map of $\mathcal{A}_{(\frp)}$-modules
$(\pi_T\circ\varpi)_{(\frp)}$,
is surjective, we can choose an $\mathcal{A}_{(\frp)}$-basis $\hat{\underline{b}}$ of $P_{(\frp)}$ with the displayed property in claim (iii). Then, as the elements $\{b_{\frp,i}\}_{i\in [r]}$ are linearly independent over $\mathcal{A}_{(\frp)}$, for each subset $\{x_i\}_{i\in [d]} \subset \mathcal{A}_{(\frp)}$, one has $x_i = 0$ for all $i \in [r]$ whenever $\sum_{i=1}^{i=d}x_i\hat b_i \in \ker(\pi_T\circ\varpi)$. In particular, since $\im(\phi) \subseteq \ker(\pi_T\circ\varpi)$, this implies the final assertion of claim (iii). 

To prove claim (iv) we assume, without loss of generality, that $P^0_{(\frp)}=P_{(\frp)}$ and that $\psi$ is the identity map. We then set $\phi_i := \hat b_i^{\ast}\circ\phi_{(\frp)}\in\Hom_{\A_{(\frp)}}(P_{(\frp)},\A_{(\frp)})$ for each index $i$, so that $\varepsilon_{\hat{\underline{b}}}=({\wedge}_{i=r+1}^{i=d}\phi_i)({\wedge}_{c=1}^{c=d}\hat b_{c})$. Then  \cite[Lem. 4.10]{bses} implies that the element ${\wedge}_{c=1}^{c=d}\hat b_{c}$ belongs to ${{{\bigcap}}}_{\A_{(\frp)}}^d P_{(\frp)}$ and thus \cite[Th. 4.19(v)]{bses} in turn implies that $\varepsilon_{\hat{\underline{b}}}$ belongs to ${{{\bigcap}}}_{\A_{(\frp)}}^r P_{(\frp)}$.

We also observe that, as the cokernel of the map $\iota_{(\frp)}$ in the exact sequence of Proposition \ref{yoneda}(ii) is $\A_{(\frp)}$-torsion-free, Lemma \ref{useful props remark}(iv) implies ${\rm Ext}^1_{\A_{(\frp)}}(\cok(\iota_{(\frp)}),\A_{(\frp)})$ vanishes. Then  \cite[Th. 4.19 (iv)]{bses} implies 
\begin{equation}\label{bigcapbigwedge}{{{\bigcap}}}^r_{\A_{(\frp)}}({^\Pi}\co^\times_{L,S,T})_{(\frp)}={\bigwedge}^r_{A}(\calF\cdot {^\Pi}\mathcal{O}_{L,S}^\times)\cap {{{\bigcap}}}_{\A_{(\frp)}}^r P_{(\frp)}\end{equation}
and so $\varepsilon_{\hat{\underline{b}}}\in {{{\bigcap}}}^r_{\A_{(\frp)}}({^\Pi}\co^\times_{L,S,T})_{(\frp)}$ if $\varepsilon_{\hat{\underline{b}}}\in {\bigwedge}^r_{A}(\calF\cdot {^\Pi}\mathcal{O}_{L,S}^\times)$.

We next claim that $\varepsilon_{\hat{\underline{b}}} = e_\pi(\varepsilon_{\hat{\underline{b}}})$.  
 %
  %
 This is true if $e'(\varepsilon_{\hat{\underline{b}}}) = 0$ for every primitive idempotent $e'$ of $A_\CC$ orthogonal to $e_\pi$. But for any such $e'$ the surjective map
$e'(\CC\cdot {_\Pi}X_{L,S}) \to e'(\CC\cdot Y_{\pi})$ induced by $\pi$ is not bijective, and so the exact sequence in Proposition \ref{yoneda}(ii) implies
\begin{align*} {\rm dim}_{\CC}(e'(\CC\cdot{\rm im}(\phi))) &= {\rm dim}_{\CC}(e'(\CC\cdot P)) - {\rm dim}_{\CC}(e'(\CC\cdot {_\Pi}X_{L,S}))\\ &< {\rm dim}_{\CC}(e'(\CC\cdot P)) - {\rm dim}_\CC(e'(\CC\cdot Y_\pi)) \\
&= {\rm dim}_{\CC}(e'(A_\CC)^d) - {\rm dim}_\CC(e'(A_\CC)^r)\\
&= (d-r){\rm dim}_{\CC}(e'A_\CC).\end{align*}
This inequality combines with \cite[Lem. 4.12]{bses} to imply $e'({\rm im}({\bigwedge}_{i=r+1}^{i=d} \phi_i))=0$, as required.

We note next that the exact triangle (\ref{Ttriangle}) implies $e_\pi$ is the sum ${\sum} e$
of all primitive central idempotents $e$ of $A$ with $e(\mathcal{F}\cdot\ker(\pi_T))=(0)$. Since the map $e_\pi(\calF\cdot {_\Pi}\mathcal{S}_S^T(L)^{\rm tr}) \to e_\pi(\calF\cdot Y_{\pi})$ induced by $\pi_T$ is bijective, our choice of basis $\hat{\underline{b}}$ implies both that $\{e_\pi(\hat b_{i})\}_{r< i \le d}$ is a basis of the $e_\pi A$-module $e_\pi(\calF\cdot {\rm im}(\phi))$ and also that $e_\pi(\calF\cdot {^\Pi}\mathcal{O}_{L,S}^\times)$ is the kernel of the map $e_\pi(\calF\cdot P) \to e_\pi{\prod}_{r< c\le d} A$ given by $(\phi_c)_c$. Applying again the general result of \cite[Lem. 4.14]{bses} in this context we may therefore deduce that there is a containment
\begin{equation}\label{key inclusion} \varepsilon_{\hat{\underline{b}}} = e_\pi(\varepsilon_{\hat{\underline{b}}}) \in e_\pi{\bigwedge}^r_{A}(\calF\cdot {^\Pi}\mathcal{O}_{L,S}^\times).\end{equation}
In particular, this combines with (\ref{bigcapbigwedge}) to imply $\varepsilon_{\hat{\underline{b}}}\in {{{\bigcap}}}^r_{\A_{(\frp)}}({^\Pi}\co^\times_{L,S,T})_{(\frp)}$, as required.

To prove the remaining claims we recall that $\D_{\A_{(\frp)}}\bigl(({_\Pi}C_{T})_{(\frp)}\bigr)$ is equal to 
$$\D_{\A_{(\frp)}}(P^\bullet_{(\frp)})=({{{\bigcap}}}^d_{\A_{(\frp)}}P_{(\frp)})\otimes_{\xi(\A_{(\frp)})}\Hom_{\xi(\A_{(\frp)})}\bigl({{{\bigcap}}}^d_{\A_{(\frp)}}P_{(\frp)},\xi(\A_{(\frp)})\bigr)=\xi(\A_{(\frp)})\cdot\beta_{\hat{\underline{b}}},$$
with $\beta_{\hat{\underline{b}}}:=(\wedge_{i=1}^{i=d}\hat{b}_i)\otimes(\wedge_{i=1}^{i=r}\hat{b}_i^*).$ To prove $\xi(\A_{(\frp)})\varepsilon_{\hat{\underline{b}}} = \W_{T,\frp}$  it is therefore enough to show 

\begin{equation}\label{epsilonbeta}\varepsilon_{\hat{\underline{b}}}=\Theta_{\underline{b}_{\pi,\frp}}(\beta_{\hat{\underline{b}}}).\end{equation}

After noting that (\ref{key inclusion}) and the definition of $\Theta_{\underline{b}_{\pi,\frp}}$ imply $\varepsilon_{\hat{\underline{b}}}$ and $\Theta_{\underline{b}_{\pi,\frp}}(\beta_{\hat{\underline{b}}})$ belong to $e_\pi{\bigwedge}^r_{A}(\calF\cdot {^\Pi}\mathcal{O}_{L,S}^\times)={\bigwedge}^r_{Ae_\pi}(e_\pi(\calF\cdot {^\Pi}\mathcal{O}_{L,S}^\times))$, and that $e_\pi(\calF\cdot {^\Pi}\mathcal{O}_{L,S}^\times)$ is a free $Ae_\pi$-module of rank $r$, the general results of \cite[Lem. 4.12 and Th. 4.19(ii)]{bses} combine to imply that (\ref{epsilonbeta}) is valid if,
for every $r$-tuple $(\varphi_i)_{i \in [r]}$ of elements of $\Hom_{\A}({^\Pi}\mathcal{O}_{L,S,T}^\times,\A)_{(\frp)}$, one has  $(\wedge_{i=1}^{i=r}\varphi_i)(\varepsilon_{\hat{\underline{b}}})=(\wedge_{i=1}^{i=r}\varphi_i)(\Theta_{\underline{b}_{\pi,\frp}}(\beta_{\hat{\underline{b}}}))$. To check this, we fix  $(\varphi_i)_{i \in [r]}$ and use the vanishing of ${\rm Ext}^1_{\A_{(\frp)}}(\cok(\iota_{(\frp)}),\A_{(\frp)})$ to choose a pre-image  $\hat\varphi_i$ of $\varphi_i$ under the `restriction through $\iota_{(\frp)}$' map $\Hom_{\A}(P,\A)_{(\frp)} \to \Hom_{\A}({^\Pi}\mathcal{O}_{L,S,T}^\times,\A)_{(\frp)}$. Then, defining $N(\phi,\{\hat\varphi_i\},\hat{\underline{b}})\in {\rm M}_d(A)$ by 
\begin{equation}\label{matrixN} N(\phi,\{\hat\varphi_i\},\hat{\underline{b}})_{ij} = \begin{cases} \hat\varphi_j(\hat b_i), &\text{if $1\le i\le d$, $1\le j\le r$}\\
                                                \phi_j(\hat b_i), &\text{if $1\le i\le d$, $r < j \le d$,}\end{cases}\end{equation} 
one obtains the required equality via the computation
\begin{equation}\label{almost2}(\wedge_{i=1}^{i=r}\varphi_i)(\varepsilon_{\hat{\underline{b}}})={\rm Nrd}_{A}(N(\phi,\{\hat\varphi_i\},\hat{\underline{b}})) = (\wedge_{i=1}^{i=r}\varphi_i)(\Theta_{\underline{b}_{\pi,\frp}}(\beta_{\hat{\underline{b}}})),\end{equation}
where the first equality follows directly from the general result \cite[Lem. 4.10]{bses} and the second from the explicit computation in \cite[Lem. 7.3.1]{dals}. 
\end{proof}

\begin{remark}\label{A2remark}{\em  If $\A$ is not locally-Gorenstein, then it is possible 
$\varepsilon_{\hat{\underline{b}}}\notin {{\bigcap}}^r_{\A_{(\frp)}}({^\Pi}\co^\times_{L,S,T})_{(\frp)}$, but the argument of Proposition \ref{yoneda}(iv) still shows $\xi(\A_{(\frp)})\varepsilon_{\hat{\underline{b}}} = \W_{T,\frp}$. To see this, write $n$ for the exponent of ${\rm Ext}^1_{\A_{(\frp)}}(\cok(\iota_{(\frp)}),\A_{(\frp)})$. Then the argument used to prove (\ref{epsilonbeta}) shows that
\begin{multline*}(\wedge_{i=1}^{i=r}\varphi_i)(n^r\cdot \varepsilon_{\hat{\underline{b}}})= (\wedge_{i=1}^{i=r}(n\cdot\varphi_i))(\varepsilon_{\hat{\underline{b}}})= {\rm Nrd}_A(N(\phi,\{n\cdot\varphi_i\},\hat{\underline{b}}))\\ = (\wedge_{i=1}^{i=r}(n\cdot\varphi_i))(\Theta_{\underline{b}_{\pi,\frp}}(\beta_{\hat{\underline{b}}}))  = (\wedge_{i=1}^{i=r}\varphi_i)(n^r\cdot \Theta_{\underline{b}_{\pi,\frp}}(\beta_{\hat{\underline{b}}})).\end{multline*} 
This implies $n^r\cdot\varepsilon_{\hat{\underline{b}}}=n^r\cdot\Theta_{\underline{b}_{\pi,\frp}}(\beta_{\hat{\underline{b}}})$ and so (\ref{epsilonbeta}) remains valid inside $e_\pi{\bigwedge}^r_{A}(\calF\cdot {^\Pi}\mathcal{O}_{L,S}^\times)$.}
\end{remark}

\subsection{The proof of Theorem \ref{main result2}}\label{proof of main result 1}

\subsubsection{}In view of (\ref{WT}) and the general result of \cite[Thm. 4.17 (iii)]{bses}, claim (i) of Theorem \ref{main result2} will follow if, for each $\frp\in {\rm Spec}(\Lambda)$, one has  
$\W_{T,\frp}\subseteq {{{\bigcap}}}^r_{\A_{(\frp)}}({^\Pi}\mathcal{O}^\times_{L,S,T})_{(\frp)}$. The latter fact, however, follows directly from the result of Proposition \ref{yoneda}(iv).

\subsubsection{} We next deduce Theorem \ref{main result2}(ii) from (the proof of) Proposition \ref{yoneda}(iv). To do this we set  $U := \mathcal{O}_{L,S,T}^{\times}$ and write $h$ for the locally-quadratic presentation of $_\Pi \mathcal{S}_S^T(L)^{\rm tr}$ given by Proposition \ref{yoneda}(ii). We fix $\frp\in {\rm Spec}(\Lambda)$ and write $h_\frp$ for the $\frp$-localisation of $h$. Then, by (\ref{WT}), it suffices to prove an equality of $\xi(\A_{(\frp)})$-modules 
\begin{equation}\label{wanted eq}\{({\bigwedge}_{i=1}^{i=r}\varphi_a)(\varepsilon)\!:\! \,\varphi_i \in \Hom_{\mathcal{A}_{(\frp)}}((^{\Pi}U)_{(\frp)},\mathcal{A}_{(\frp)}),\,\varepsilon\in\W_{T,\frp}\} =
 {\rm Fit}^{r}_{\mathcal{A}_{(\frp)}}(h_\frp).\end{equation}
The $\A_{(\frp)}$-modules $P^0_{(\frp)}$ and $P_{(\frp)}$ are both free of rank $d$ and so, after appropriately modifying $h_\frp$,  we can assume $P^0_{(\frp)}=P_{(\frp)}$. We then recall that, for any basis $\hat{\underline{b}}$ of $P_{(\frp)}$ as in Proposition \ref{yoneda}(iii), the element $\varepsilon_{\hat{\underline{b}}}$ generates $\W_{T,\frp}$ and, for any $\varphi_i \in \Hom_{\mathcal{A}}(^{\Pi}U,\mathcal{A})_{(\frp)}$, satisfies (\ref{almost2}). The left hand side of (\ref{wanted eq}) is thus equal to
$$\xi(\A_{(\frp)})\cdot\{{\rm Nrd}_A(N(\phi,\{\hat\varphi_i\},\hat{\underline{b}})):\,\varphi_i \in \Hom_{\mathcal{A}_{(\frp)}}((^{\Pi}U)_{(\frp)},\mathcal{A}_{(\frp)})\},$$
with each matrix $N(\phi,\{\hat\varphi_i\},\hat{\underline{b}})$ as in (\ref{matrixN}). Now the final assertion of Proposition \ref{yoneda}(iii) implies the first $r$ columns of the matrix, with respect to the basis $\hat{\underline{b}}$, of the endomorphism $\phi_{(\frp)}$ of $P_{(\frp)}$  are equal to $0$. Hence, as $\varphi_i$ varies over $\Hom_{\mathcal{A}}(^{\Pi}U,\mathcal{A})_{(\frp)}$, the matrices $N(\phi,\{\hat\varphi_i\},\hat{\underline{b}})$ account for all of the matrices which both occur in the definition of
${\rm Fit}^{r}_{\mathcal{A}_{(\frp)}}(h_\frp)$ and have non-zero reduced norm. The latter Fitting invariant is therefore equal to the last displayed expression, as suffices to prove the required equality (\ref{wanted eq}).

Before proving the second assertion of claim (ii), we note that $({{\bigcap}}^{r}_{\A}{^{\Pi}}\co^\times_{L,S,T})^{e_\pi}/\mathcal{W}_T$ and  
$\xi(\mathcal{A})^{e_\pi}/{\rm Fit}_{\mathcal{A}}^r(h)$ are both finite, the former by Theorem \ref{main result} and the latter as a consequence of the definition of $e_\pi$ and the argument of \cite[Th. 3.20(iv)]{bses}. It is therefore enough to prove the existence of a perfect bilinear pairing after localisation at $\frp \in {\rm Spec}(\Lambda)$. In the rest of this section we thus fix $\frp$ and set $\mathcal{B} := \xi(\mathcal{A}_{(\mathfrak{p})})$. For a $\mathcal{B}$-module $X$, we set $X^\diamond := \Hom_{\mathcal{B}}(X,\mathcal{B})$ and $X^\wedge := \Hom_{\mathcal{B}}(X,\zeta(A)/\mathcal{B})$, both regarded as $\mathcal{B}$-modules in the natural way. For any idempotent $e$ of $\zeta(A)$ we also set $Xe := \{(1\otimes x)e: x \in X\}$ where $1\otimes x$ is the image of $x$ in $\mathcal{F}\otimes_\Lambda X$, and  write $X^e$ for the submodule $\{x \in X: e(1\otimes x) = 1\otimes x\}$ of $X$. As a first step, we establish some useful facts about  dual modules. 

\begin{lemma}\label{dual-frac} Assume $\mathcal{B}$ is a Gorenstein $\Lambda_{(\frp)}$-order and fix an  
idempotent $e$ of $\zeta(A)$. 
\begin{itemize}
\item[(i)] If $X$ is $\Lambda_{(\mathfrak{p})}$-free, then the natural maps $X \to X^{\diamond\diamond}$ and $(Xe)^\diamond\to (X^\diamond)^e$ are bijective.
\item[(ii)] If $X$ is finite, then the natural map $X \to X^{\wedge\wedge}$ is bijective.
\end{itemize}
\end{lemma}
\begin{proof} The first assertion of claim (i) is true since, as $\mathcal{B}$ is Gorenstein, every $\mathcal{B}$-lattice is reflexive (cf. \cite[Th. 6.2]{bass}). The second assertion is true since the maps $(Xe)^\diamond\to (X^\diamond)^e$ and $(X^\diamond)^e\to (Xe)^\diamond$ given by $\theta\mapsto (x \mapsto \theta(xe))$ and $\phi\mapsto (xe \mapsto \phi(x))$ are mutually inverse. To prove claim (ii) we fix a surjective map of $\mathcal{B}$-modules $\mathcal{B}^d\rightarrow X$ (with $d \in \mathbb{N}$) and write $Y$ for its kernel. We then need only  consider the exact commutative diagram
\[ \begin{CD}
0 @> >> Y^{\diamond\diamond} @> >> (\mathcal{B}^d)^{\diamond\diamond} @>  >>  X^{\wedge\wedge}@> >> 0\\
@. @A AA @A AA @A AA \\
0 @> >> Y @> >> \mathcal{B}^d @>  >>  X@> >> 0.
\end{CD}\]
Here all vertical arrows are the canonical maps and the upper row is obtained by (twice) applying the functor $\Hom_\mathcal{B}(-,\mathcal{B})$ to the exact sequence
 $0\rightarrow Y\rightarrow \mathcal{B}^d\rightarrow X\rightarrow 0$ and noting that $\Ext_{\mathcal{B}}^1(X,\mathcal{B})$ identifies with $X^\wedge$. Then, since claim (i) implies the first two vertical maps are bijective, the third vertical map must also be bijective, as required. \end{proof}

By \cite[Lem. 4.16]{bses}, the $\mathcal{B}$-submodule of ${\bigwedge}^r_{A^{\rm op}} (\calF\cdot 
\Hom_{\A_{(\frp)}}({^\Pi}U,\A))$ given by   
\[ {\bigwedge}^r_{\mathcal{A}_{(\frp)}^{\rm op}} \Hom_{\A}({^\Pi}U,\A)_{(\frp)} := \mathcal{B}\cdot \{ {\wedge}_{i \in [r]}\varphi_i: \varphi_i \in \Hom_{\A}({^\Pi}U,\A)_{(\frp)}\}\]
is finitely generated. Further, by combining the definition of reduced Rubin lattices with (the first assertion of) Lemma \ref{dual-frac}(i), one finds this lattice is equal to $\bigl({\bigcap}^{r}_{\A}{^{\Pi}}U\bigr)_{(\frp)}^\diamond$. 

We now write $Q$ for the finite $\mathcal{B}$-module $\bigl({\bigcap}^{r}_{\A}{^{\Pi}}U\bigr)_{(\frp)}^{e_\pi}/\mathcal{W}_{T,\frp}$. Then, by applying the functor $\Hom_\mathcal{B}(-,\mathcal{B})$ to the tautological exact sequence %
\[ 0 \to \mathcal{W}_{T,\frp} \to \bigl({\bigcap}^{r}_{\A}{^{\Pi}}U\bigr)_{(\frp)}^{e_\pi} \to Q \to 0,\]
and using (the second assertion of) Lemma \ref{dual-frac}(i) with $X = \bigl({\bigcap}^{r}_{\A}{^{\Pi}}U\bigr)_{(\frp)}$ and $e=e_\pi$, one obtains the upper row of the exact commutative diagram
\[\begin{CD} 0 @> >> \bigl( {\bigwedge}^r_{\mathcal{A}_{(\frp)}^{\rm op}} \Hom_{\A}({^\Pi}U,\A)_{(\frp)}\bigr)e_\pi @> >> \mathcal{W}_{T,\frp}^\diamond @> >> Q^\wedge @> >> 0\\
& & @V \epsilon VV @V \epsilon' VV \\
0 @> >> {\rm Fit}_\mathcal{A}^r(h)_{(\frp)} @> >> \mathcal{B}^e @> >> \mathcal{B}^e/{\rm Fit}_\mathcal{A}^r(h)_{(\frp)} @> >> 0.\end{CD}\]
Here the lower exact sequence is obvious and both $\epsilon$ and $\epsilon'$ are induced by restriction of the map ${\bigwedge}^r_{A^{\rm op}} (\calF\cdot 
\Hom_{\A}({^\Pi}U,\A))\to \zeta(A)$ that sends each $\Phi$ to $\Phi(\varepsilon_{\hat{\underline{b}}})$. In particular, since $\mathcal{W}_{T,\frp}=\mathcal{B}\varepsilon_{\hat{\underline{b}}}$ is a free $\mathcal{B}e_\pi$-module of rank one, the first claim of Theorem \ref{main result2}(ii) (as proved above) implies $\epsilon$ is bijective and, since $(\mathcal{B}e_\pi)^\diamond$ is isomorphic to $(\mathcal{B}^\diamond)^{e_\pi} = \mathcal{B}^{e_\pi}$ (by Lemma \ref{dual-frac}(i)), the map $\epsilon'$ is also bijective. From the above diagram, we can therefore deduce the existence of an induced isomorphism of $\mathcal{B}$-modules 
\[ \iota: \bigl(\bigl({\bigcap}^{r}_{\A}{^{\Pi}}U\bigr)_{(\frp)}^{e_\pi}/\mathcal{W}_{T,\frp}\bigr)^\wedge = Q^\wedge \cong \mathcal{B}^{e_\pi}/{\rm Fit}_\mathcal{A}^r(h)_{(\frp)}.\]
This isomorphism in turn gives rise to an $\mathcal{B}$-bilinear pairing of finite modules
\[ \bigl(\bigl({\bigcap}^{r}_{\A}{^{\Pi}}U\bigr)_{(\frp)}^{e_\pi}/\mathcal{W}_{T,\frp}\bigr) \times 
\mathcal{B}^{e_\pi}/{\rm Fit}_\mathcal{A}^r(h)_{(\frp)} \to \zeta(A)/\mathcal{B}, \quad (u,v)\mapsto (\iota^{-1}(v))(u)\]
and Lemma \ref{dual-frac}(ii) implies that this pairing is perfect. 

\subsubsection{} In this section we reduce the proof of Theorem \ref{main result2}(iii) to a claim which will be verified in \S\ref{now turn}. Throughout we assume the notation and hypotheses of Theorem \ref{main result2}(iii). For an $\mathcal{A}$-module $M$ we set 
$M^* := \Hom_{\mathcal{A}}(M,\mathcal{A})$. We first require the two following lemmas. 

\begin{lemma}\label{tech21} Fix $\varepsilon\in \W_T$ and $m \in \mathbb{N}$. Then there exists $n_m\in \mathbb{N}$ with the following property: if $\{\varphi_i\}_{i \in [r]}$ and $\{\varphi'_i\}_{i \in [r]}$ are subsets of $(^{\Pi}\mathcal{O}_{L,S,T}^\times)^*$ with $\varphi_i -\varphi'_i \in n_m\!\cdot (^{\Pi}\mathcal{O}_{L,S,T}^\times)^*$ for all $i\in [r]$, then $(\wedge_{i=1}^{i=r}\varphi_i)(\varepsilon) - (\wedge_{i=1}^{i=r}\varphi'_i)(\varepsilon)\in m\cdot\mathcal{A}.$
\end{lemma}

\begin{proof} In this argument we fix a decomposition ${\prod}_{j\in J}A_j$ 
of $A$ as a product of simple algebras, with $1={\sum}_{j\in J}e_j$ the corresponding decomposition of the identity of $A$ as a sum of primitive central idempotents. Fix a number field $E$ that splits each algebra $A_j$, set $A_j':= E\cdot A_{j}$, fix a simple $A_j'$-module $V_j$ and set $d_j:= {\rm dim}_{E}(V_j)$. Write $\mathcal{O}$ for the ring of integers of $E$ and $\mathcal{A}'_j$ for the $\mathcal{O}$-order in $A'_j$ generated by the image of $\mathcal{A}$ in $A_j$. Then, if necessary after increasing $E$, we can assume (compatibly with  the normalisation of Remark \ref{normalization rem}) that for each $j$ there exists a full $\mathcal{A}_j'$-sublattice $T_j$ in $V_j$ that is free over $\mathcal{O}$ and such that the elements $\{v_k\}_{k \in [d_j]}$ used in the definition \cite[(4.2.3)]{bses} of the $A_j$-th component of each reduced exterior power $\wedge_{i=1}^{i=r}\varphi_i$ are an $\mathcal{O}$-basis of $T_j$. With respect to these choices, each element $\wedge_{i=1}^{i=r}\varphi_i$ belongs to
\[ \Gamma:= {\bigoplus}_{j\in J} {\bigwedge}_{\mathcal{O}}^{rd_j}(T_j\otimes_{(\mathcal{A}_j')^{\rm op}}\Hom_{\mathcal{A}_j'}(e_j(\mathcal{O}\cdot(^{\Pi}\mathcal{O}_{L,S,T}^\times)),\mathcal{A}_j')) \subset E\cdot {{\bigwedge}}^r_{A^{\rm op}}\Hom_A(\QQ\cdot ^{\Pi}\mathcal{O}_{L,S,T}^\times,A).\]
In addition, since $\Gamma$ is a finitely generated $\mathcal{O}$-module and $\mathcal{O}\cdot\mathcal{A}$ has finite index in ${\prod}_{j\in J}\mathcal{A}'_j$, there exists an integer $n_m$ with the property that $\lambda(\varepsilon) \in m (\mathcal{O}\cdot \mathcal{A})$ for all $\lambda\in n_m\cdot \Gamma$.

Now, by the stated assumptions, one has $\wedge_{i=1}^{i=r}\varphi_i - \wedge_{i=1}^{i=r}\varphi'_i\in n_m\cdot\Gamma$ and hence (by the above argument) also $(\wedge_{i=1}^{i=r}\varphi_i)(\varepsilon) - (\wedge_{i=1}^{i=r}\varphi'_i)(\varepsilon)\in m(\mathcal{O}\cdot\mathcal{A})$. This then implies the claimed congruence since $\wedge_{i=1}^{i=r}\varphi_i$ and $\wedge_{i=1}^{i=r}\varphi'_i$ belongs to $A$ and one has $A\cap m(\mathcal{O}\cdot\mathcal{A}) = m\cdot\mathcal{A}$. \end{proof}

\begin{lemma}\label{tech2}  For each $\{\varphi_i\}_{i\in [r]}\subset (^{\Pi}\mathcal{O}_{L,S,T}^\times)^*$ and $n \in \mathbb{N}$, there exists a subset $\{\varphi_i'\}_{i \in [r]}$ of $(^{\Pi}\mathcal{O}_{L,S,T}^\times)^*$ with the following properties.
\begin{itemize}
\item[(i)] For all $i\in [r]$ one has $\varphi_i' - \varphi_i\in n\cdot (^{\Pi}\mathcal{O}_{L,S,T}^\times)^*$.
\item[(ii)] Write $\rho_{S,S'}$ for the restriction map $(^{\Pi}\mathcal{O}_{L,S,T}^\times)^*\to (^{\Pi}\mathcal{O}^\times_{L,S',T})^*$, with $S'$ as in Theorem \ref{main result2}(iii). Then the $\mathcal{A}$-module 
 generated by $\{\rho_{S,S'}(\varphi_i')\}_{i\in [r]}$ is free of rank $r$.
    \end{itemize}
\end{lemma}
\begin{proof} The explicit choice of $S'$ in Theorem \ref{main result2}(iii) implies that we may choose a free $\mathcal{A}$-submodule $\mathfrak{F}$ of $(^{\Pi}\mathcal{O}^\times_{L,{S'},T})^*$ of rank $r$. We then choose a subset $\{f_i\}_{i\in [r]}$ of $(^{\Pi}\mathcal{O}_{L,S,T}^\times)^*$ for which $\{\rho_{S,{S'}}(f_i)\}_{i\in [r]}$ is an $A$-basis of $\QQ\cdot\mathfrak{F}$. For any integer $m$ we set $\varphi_{i,m} := \varphi_i + mn f_i$ and note it suffices to show that for any large enough $m$, the elements $\{\rho_{S,{S'}}(\varphi_{i,m})\}_{i\in [r]}$ are linearly independent over $A$.

Consider the composite map of $\mathcal{A}$-modules $\mathfrak{F} \to \QQ\cdot(^{\Pi}\mathcal{O}^\times_{L,{S'},T})^*\to \QQ\cdot \mathfrak{F}$ where the first arrow sends each $\rho_{S,{S'}}(f_i)$ to $\rho_{S,{S'}}(\varphi_{i,m})$ and the second is induced by a choice of $A$-equivariant section to the projection $\QQ\cdot (^{\Pi}\mathcal{O}^\times_{L,{S'},T})^*\to \QQ\cdot ((^{\Pi}\mathcal{O}^\times_{L,{S'},T})^*/\mathfrak{F})$. Then, with respect to the basis $\{\rho_{S,{S'}}(f_i)\}_{i\in [r]}$, this map is represented by a matrix of the form $M + mn  I_r$ for a matrix $M$ in ${\rm M}_r(A)$ that is independent of $m$. In particular, if $m$ is large enough to ensure that $-mn$ is not an eigenvalue of the image of $M$ in any of the simple algebra components of ${\rm M}_r(A_\CC)$, then the composite homomorphism is injective and so the elements $\{\rho_{S,{S'}}(\varphi_{i,m})\}_{i \in [r]}$ are linearly independent over $A$, as required.\end{proof}

The next result reduces Theorem \ref{main result2}(iii) to a claim that will be verified in \S\ref{now turn}. 
For a $G$-module $M$ we endow  $M^\vee:=\Hom_\ZZ(M,\QQ/\ZZ)$ with the contragredient action of $G$.

\begin{proposition}\label{local reduction} Theorem \ref{main result2}(iii) is true if for any  
\begin{itemize}
\item[-] reduced exterior product $\Phi = {{\bigwedge}}_{i=1}^{i=r}\varphi_i$ for which $\{\rho_{S,S'}(\varphi_i)\}_{i\in[r]}$ spans a free $A$-module of rank $r$,
\item[-] $\frp\in{\rm Spec}(\Lambda)$ and $\varepsilon_{\hat{\underline{b}}}$ as in Proposition \ref{yoneda}(iv),
\item[-] $a_\frp\in \delta(\A_{(\frp)})\cdot{\rm Ann}_{\A}({\rm Tor}^{G}_1(\check\Pi,\Hom_\ZZ(\mathcal{O}_{L,S',T}^\times,\ZZ))$, 
\end{itemize}
the product $a_\frp\cdot\Phi(\varepsilon_{\hat{\underline{b}}})$ belongs to $\A_{(\frp)}$ and annihilates $({_{\check\Pi}}({\rm Cl}^T_{S'}(L)^\vee))_{(\frp)}$.
\end{proposition}

\begin{proof} We first show that it suffices to prove Theorem \ref{main result2}(iii) for elements $\Phi$ as above.
To do this we fix an arbitrary subset $\{\varphi_i\}_{i\in[r]}$ of $(^{\Pi}\mathcal{O}_{L,S,T}^\times)^*$. We also fix $m \in \mathbb{N}$ so that $m\cdot \delta(\mathcal{A})\subseteq |{\rm Cl}_{S'}^T(L)|\cdot \mathcal{A}$. We apply Lemma \ref{tech21} to $m$ and so obtain an integer $n_m$ that can be substituted  into Lemma \ref{tech2} to give a subset $\{\varphi'_i\}_{i\in[r]}$ of $(^{\Pi}\mathcal{O}_{L,S,T}^\times)^*$ for which one has
$(\wedge_{i=1}^{i=r}\varphi_i)(\varepsilon) - (\wedge_{i=1}^{i=r}\varphi'_i)(\varepsilon) \in m\cdot\mathcal{A}$. For any $a\in \delta(\mathcal{A})$ and any additive map $\epsilon:\mathcal{A} \to \ZZ$, one therefore has
 \[  \epsilon_G\bigl(\iota_\A\bigl(a\cdot(\wedge_{i=1}^{i=r}\varphi_i)(\varepsilon)\bigr)\cdot {\rm pr}_{\check\Pi}\bigr) - \epsilon_G\bigl(\iota_\A\bigl(a\cdot(\wedge_{i=1}^{i=r}\varphi'_i)(\varepsilon)\bigr)\cdot {\rm pr}_{\check\Pi}\bigr) \in \epsilon(m\cdot \delta(\mathcal{A}))\cdot \ZZ[G].\]
Since our choice of $m$ implies $\epsilon(m\cdot \delta(\mathcal{A}))\subseteq \epsilon(|{\rm Cl}_{S'}^T(L)|\cdot\mathcal{A})
\subseteq |{\rm Cl}_{S'}^T(L)|\cdot \ZZ[G]$, this containment implies Theorem \ref{main result2}(iii) is valid for $\wedge_{i=1}^{i=r}\varphi_i$ if it is valid for  $\wedge_{i=1}^{i=r}\varphi'_i$. This verifies the claimed reduction since Lemma \ref{tech2}(ii) implies $\{\rho_{S,S'}(\varphi'_i)\}_{i\in[r]}$ spans a free $A$-module of rank $r$.

To complete the proof we next observe that for any $\frp\in{\rm Spec}(\Lambda)$ one has $\delta(\A_{(\frp)})\cdot\xi(\A_{(\frp)})=\delta(\A_{(\frp)})$ (by \cite[Lem. 3.7(iv)]{bses}) and $\W_{T,\frp}=\xi(\A_{(\frp)})\cdot\varepsilon_{\hat{\underline{b}}}$ (by Proposition \ref{yoneda}(iv)). Hence, the stated hypotheses would imply that for any element $\varepsilon_{\frp}\in \W_{T,\frp}$, the product $a_\frp\cdot\Phi(\varepsilon_{\frp})$ belongs to $\A_{(\frp)}$ and annihilates ${_{\check\Pi}}({\rm Cl}^T_{S'}(L)^\vee)_{(\frp)}$. Since $\delta(\A)$ is defined to be ${{{\bigcap}}}_{\frp}\delta(\A_{(\frp)})$ (cf. \cite[Def. 3.6]{bses}), the second equality in (\ref{WT}) would therefore imply that $a\cdot\Phi(\varepsilon)$ belongs to $\A$ and annihilates ${_{\check\Pi}}({\rm Cl}^T_{S'}(L)^\vee)$ for any $a$ and any $\varepsilon$ as in the statement of Theorem \ref{main result2}(iii).
By a routine computation, as in the proof of \cite[Prop. 3.4(ii)]{bses2}, this last fact implies $\epsilon_{G}\bigl(\iota_\A(a\cdot\Phi(\varepsilon))\cdot {\rm pr}_{\check\Pi}\bigr)$ belongs to $\ZZ[G]$ and annihilates ${\rm Cl}_{S'}^T(L)$, as required. 
\end{proof}

\subsubsection{}As the final step before proving Theorem \ref{main result2}(iii), we establish some useful properties of Weil-\'etale complexes. Before stating this result we recall that the `($S$-relative $T$-trivialised) integral dual Selmer group $\mathcal{S}_S^T(L)$ for $\GG_m$ over $L$' is defined in \cite[Def. 2.1]{bks} and lies in a 
%
canonical exact sequence
\begin{equation}\label{selmer lemma I} 0 \to {\rm Cl}_{S}^T(L)^\vee \to \mathcal{S}_S^T(L) \to \Hom_{\ZZ}(\mathcal{O}^\times_{L,S,T},\ZZ)\to 0.\end{equation}
%
%

\begin{lemma}\label{dual complex lemma} 
The complex $(_{\Pi}C_{L,S,T})^*:=R\Hom_\A({_\Pi}C_{L,S,T},\A)[-1]$ belongs to $D^{\rm lf}(\mathcal{A})$, is acyclic outside degrees zero and one, and such that there are natural identifications
\[ H^{0}((_{\Pi}C_{L,S,T})^*) = ({_{\Pi}}X_{L,S})^*\,\,\text{ and }\,\,H^{1}((_{\Pi}C_{L,S,T})^*) = {_{\check\Pi}}\mathcal{S}_S^T(L).\]
In addition, there are canonical surjections ${_{\check\Pi}}\mathcal{S}_\Sigma^T(L)\to (^{\Pi}\mathcal{O}_{L,\Sigma,T}^\times)^*$ for $\Sigma \in \{S,S'\}$ and $f_{S,S'}:{_{\check\Pi}}\mathcal{S}_S^T(L) \rightarrow {_{\check\Pi}}\mathcal{S}_{S'}^T(L)$ that make the following square commute:
\begin{equation*}
\xymatrix{ {_{\check\Pi}}\mathcal{S}_S^T(L) \ar@{->}[d]  \ar@{->}[r]^{f_{S,S'}}    & {_{\check\Pi}}\mathcal{S}_{S'}^T(L)  \ar@{->}[d]\\
(^{\Pi}\mathcal{O}_{L,S,T}^\times)^*   \ar@{->}[r]^{\rho_{S,S'}} &  (^{\Pi}\mathcal{O}_{L,S',T}^\times)^*.
} 
\end{equation*}
\end{lemma}

\begin{proof} Set $C := C_{L,S,T}$ and $C^* := R\Hom_\Z(C,\Z)[-1]$ in $D^{\rm lf}(\ZZ[G])$. Then Proposition \ref{yoneda}, with $\Pi=\ZZ[G]$, implies $C$ is represented by a complex $P^\bullet$ of finitely generated, locally-free $G$-modules $P^0\to P$, where the first module is placed in degree zero, and hence that ${_{\check\Pi}}C^* := \check\Pi\otimes_{\ZZ[G]}^\mathbb{L}C^*$ is represented by the complex ${_{\check\Pi}}\Hom_\ZZ(P^\bullet,\ZZ)[-1]$. This implies ${_{\check\Pi}}C^*$ belongs to $D^{\rm lf}(\mathcal{A})$, is acyclic outside degrees zero and one and $H^1({_{\check\Pi}}C^*) = {_{\check\Pi}}H^1(C^*) ={_{\check\Pi}}\mathcal{S}_S^T(L)$, where the second equality is induced by the identification $H^1(C^*) = \mathcal{S}_S^T(L)$ in \cite[Prop. 2.4(iii)]{bks}.

We next observe that ${_{\check\Pi}}C^*$ is naturally isomorphic to the complex $(_{\Pi}C)^*$. 
This is true because for any finitely generated projective $G$-module $Q$ there are natural isomorphisms
\begin{multline*}\Hom_\mathcal{A}({_{\Pi}}Q,\mathcal{A}) = \Hom_\mathcal{A}(H_0(G,{\Pi}\otimes_\ZZ Q),\mathcal{A}) \cong H^0(G,\Hom_\mathcal{A}({\Pi}\otimes_\ZZ Q,\mathcal{A})) \\ \cong H_0(G,\Hom_\mathcal{A}({\Pi}\otimes_\ZZ Q,\mathcal{A})) \cong H_0(G,\check\Pi\otimes_\ZZ \Hom_\ZZ(Q,\ZZ))\cong {_{\check\Pi}}(\Hom_\ZZ(Q,\ZZ)),\end{multline*}
where the second isomorphism is induced by the action of $T_G$ (and the fact that the $G$-module $\Hom_\mathcal{A}({\Pi}\otimes_\ZZ Q,\mathcal{A})\cong \check\Pi\otimes_\ZZ \Hom_\ZZ(Q,\ZZ)$ is cohomologically-trivial).

Finally, there is a spectral sequence ${\rm Ext}_{\mathcal{A}}^a(H^{b}({_{\Pi}}C),\mathcal{A}) \Rightarrow H^{a-b+1}((_{\Pi}C)^*)$ which combines with Lemma \ref{useful props remark} to imply that in each degree $i$ there is a natural exact sequence
\begin{equation}\label{spectral} 0 \to \Hom_\mathcal{A}(H^{-i+2}({_{\Pi}}C),A/\mathcal{A}) \to H^i((_{\Pi}C)^*) \to H^{-i+1}({_{\Pi}}C)^* \to 0\end{equation}
This sequence combines with the descriptions in Lemma \ref{Picoh} to give the claimed identification $H^0((_{\Pi}C)^*) = ({_{\Pi}}X_{L,S})^*$  and surjection ${_{\check\Pi}}\mathcal{S}_S^T(L)=H^1({_{\check\Pi}}C^*)=H^{1}((_{\Pi}C)^*)\to (^{\Pi}\mathcal{O}_{L,S,T}^\times)^*$.

The same argument applies with $C$ replaced by the complex $C_{L,S',T}$ and, in particular, the corresponding sequence of the form (\ref{spectral}) induces a surjection ${_{\check\Pi}}\mathcal{S}_{S'}^T(L)\to (^{\Pi}\mathcal{O}_{L,S',T}^\times)^*$.

Now the canonical morphism $C_{L,S,T}^*\to C^*_{L,S',T}$ in $D^{\rm lf}(\ZZ[G])$ from \cite[Prop. 2.4(ii)]{bks} induces a surjective map of $G$-modules $\mathcal{S}_S^T(L) \rightarrow \mathcal{S}_{S'}^T(L)$ and, by construction, its base-change $f_{S,S'}$ makes the claimed square commute.
\end{proof}

\begin{remark}\label{factorsthroughSel}{\em The above argument also shows that for any representative $P^\bullet=[P^0\to P]$ of ${_\Pi}C$ in $D^{\rm lf}(\A)$ as in Proposition \ref{yoneda}, the composite projection map 
\[ P^{0,*}\to H^1(({_\Pi}C)^*)= {_{\check\Pi}}\mathcal{S}_S^T(L) \to (^{\Pi}\mathcal{O}_{L,S,T}^\times)^*\]
is the linear dual of the injective map $^{\Pi}\mathcal{O}_{L,S,T}^\times\to P^0$ in Proposition \ref{yoneda}(ii). }
\end{remark}
%

\subsubsection{}\label{now turn} We now finally complete the proof of Theorem \ref{main result2}(iii). To do this, following Proposition \ref{local reduction}, we fix $\Phi=\wedge_{j=1}^{j=r}\varphi_j$ as in the latter result, $\frp\in {\rm Spec}(\Lambda)$ and $\varepsilon_{\hat{\underline{b}}}$ as in Proposition \ref{yoneda}(iv). We then choose a pre-image $\widetilde\varphi_j$ of each map $\varphi_j$ under the surjective map ${_{\check\Pi}}\mathcal{S}_S^T(L)\to (^{\Pi}\mathcal{O}_{L,S,T}^\times)^*$ in Lemma \ref{dual complex lemma} and write $\mathcal{E}_\Phi$ for the $\mathcal{A}$-module generated by $\{\widetilde\varphi_j\}_{j\in[r]}$. 

The following observation about these elements is key to our argument.

\begin{proposition}\label{ltc prop} 
There is a quadratic presentation $h_\Phi$ of the $\mathcal{A}_{(\frp)}$-module $({_{\check\Pi}}\mathcal{S}_S^T(L)/\mathcal{E}_\Phi)_{(\frp)}$ such that $\Phi(\varepsilon_{\hat{\underline{b}}})\in {\rm Fit}_{\mathcal{A}_{(\frp)}}^0(h_\Phi)$.\end{proposition}

\begin{proof} We use the existence of an exact triangle in $D^{\rm lf}(\mathcal{A}_{(\frp)})$ of the form
\begin{equation}\label{triangle} \mathcal{A}_{(\frp)}^{\oplus r,\bullet} \xrightarrow{\theta} (_{\Pi}C)^*_{(\frp)} \xrightarrow{\theta'} D^\bullet \to \mathcal{A}_{(\frp)}^{\oplus r,\bullet}[1]. \end{equation}
Here $\mathcal{A}_{(\frp)}^{\oplus r,\bullet}$ denotes the complex $\mathcal{A}_{(\frp)}^{\oplus r}[0]\oplus \mathcal{A}_{(\frp)}^{\oplus r}[-1]$ with $\mathcal{A}_{(\frp)}^{\oplus r}$ the direct sum of $r$ copies of $\mathcal{A}_{(\frp)}$. Also, writing $\{c_i\}_{i\in[r]}$ for the canonical basis of $\mathcal{A}_{(\frp)}^{\oplus r}$, the morphism $\theta$ is uniquely specified by the condition that for each $j\in\{0,1\}$ one has

\[ H^j(\theta)(c_i) = \begin{cases} b_{\frp,i}^*\in Y_{\pi,(\frp)}^* \subset ({_\Pi}X_{L,S})_{(\frp)}^* = H^0((_{\Pi}C)^*_{(\frp)}), &\text{if $j=0$}\\
                 \widetilde\varphi_i\in {_{\check\Pi}}\mathcal{S}_S^T(L)_{(\frp)} = H^1((_{\Pi}C)^*_{(\frp)}), &\text{if $j=1$},\end{cases}\]
with the inclusion in the first line induced by the dual of $\pi_{(\frp)}$, and with $\underline{b}_{\pi,\frp}=\{b_{\frp,i}\}_{i\in[r]}$ the fixed $\A_{(\frp)}$-basis of $Y_{\pi,(\frp)}$, a fixed lift $\hat{\underline{b}}=\{\hat{b}_i\}_{i\in[r]}$ of which has been used to construct the element $\varepsilon_{\hat{\underline{b}}}$ in Proposition \ref{yoneda}(iv). With this definition the exact cohomology sequence of the triangle (\ref{triangle}) implies $D^\bullet$ is acyclic outside degrees zero and one and induces identifications
\begin{equation}\label{Dcoh} H^0(D^\bullet) = \ker(\pi_{(\frp)})^*\,\,\text{ and }\,\, H^1(D^\bullet) =                                  ({_{\check\Pi}}\mathcal{S}_S^T(L)/\mathcal{E}_\Phi)_{(\frp)}.\end{equation}
%

We now consider the following diagram of $\A_{(\frp)}$-modules.

\begin{equation*}
\xymatrix{ \mathcal{A}_{(\frp)}^{\oplus r}  \ar@{->}[rr]^{{\rm id}}  & & \mathcal{A}_{(\frp)}^{\oplus r} &  & \\
\mathcal{A}_{(\frp)}^{\oplus r} \ar@{->}[u]^{{\rm id}}  \ar@{->}[rr]^{(0,\theta^0)} &  &  \mathcal{A}_{(\frp)}^{\oplus r}\oplus P_{(\frp)}^*  \ar@{->>}[u]_{(-{\rm id},\mu)} \ar@{->}[rr]^{(\theta^1,\phi_{(\frp)}^*)}  & & P_{(\frp)}^*\\
& &  P_{(\frp)}^* \ar@{^{(}->}[u]^{(\mu,{\rm id})} \ar@{->}[rr]^{\nu}  & & P_{(\frp)}^* \ar@{->}[u]_{{\rm id}} .
} 
\end{equation*}
Here $P^0\stackrel{\phi}{\longrightarrow} P$ is the representative of ${_\Pi}C$ in $D^{\rm lf}(\A)$ as in Proposition \ref{yoneda} that we have used in the construction of given element $\varepsilon_{\hat{\underline{b}}}$, after assuming that $P^0_{(\frp)}=P_{(\frp)}$ and fixing a lift $\hat{\underline{b}}=\{\hat{b}_i\}_{i\in [r]}$ of $\underline{b}$ as in Proposition \ref{yoneda}(iii). In addition, we have set
$$\mu(\hat{b}_i^*):=\begin{cases}c_i,\,\,\,\,\,\,1\leq i\leq r,\\ 0,\,\,\,\,\,\,\,\,r<i\leq d,\end{cases}\,\,\text{ and }\,\,\nu(\hat{b}_i^*):=\begin{cases}\theta^1(c_i),\,\,\,\,\,\,1\leq i\leq r,\\ \phi_i,\,\,\,\,\,\,\,\,\,\,\,\,\,\,\,r<i\leq d.\end{cases}$$
The above diagram then commutes because the choice of basis $\hat{\underline{b}}$ implies (via the final assertion of Proposition \ref{yoneda}(iii)) that, for each $i\in [r]$, one has $\phi_i:=\hat{b}_i^*\circ\phi_{(\frp)}=\phi_{(\frp)}^*(\hat{b}_i^*)=0$. 

Now $(_{\Pi}C)^*_{(\frp)}$ is represented by the complex $P^*_{(\frp)}\to P_{(\frp)}^*$, in which the first term is in degree zero and the differential is $\phi_{(\frp)}^*$. The central row of the above diagram thus represents the mapping cone of the morphism $\theta$ in (\ref{triangle}) and so is isomorphic in $D^{\rm lf}(\A_{(\frp)})$ to $D^\bullet$. Hence, since the upper row of the diagram is an acyclic complex, its commutativity and the exactness of its columns implies $D^\bullet$ is isomorphic in $D^{\rm lf}(\A)$ to the complex given by the lower row of the diagram. From (\ref{Dcoh}) we deduce the existence of a quadratic presentation of $\A_{(\frp)}$-modules
$$P_{(\frp)}^*\stackrel{\nu}{\longrightarrow}P_{(\frp)}^*\longrightarrow ({_{\check\Pi}}\mathcal{S}_S^T(L)/\mathcal{E}_\Phi)_{(\frp)}\to 0.$$

It is thus enough to show $\Phi(\varepsilon_{\hat{\underline{b}}}) = {\rm Nrd}_A(N)$, with $N$ the matrix of $\nu$ with respect to the basis $\hat{\underline{b}}^*=\{\hat{b}_i^*\}_{i\in [d]}$ of $P^*_{(\frp)}$. To do this, we note Lemma \ref{useful props remark}(iv) implies ${\rm Ext}^1_{\A_{(\frp)}}({\rm cok}(\iota_{(\frp)}),\A_{(\frp)})$ vanishes, where $\iota$ is the canonical injection ${^{\Pi}}\co^\times_{L,S,T}=H^0({_\Pi}C)\hookrightarrow P^0$. As in the proof of Proposition \ref{yoneda}(iv), we can therefore fix pre-images $\hat{\varphi}_j$ of each $\varphi_j$ under the `restriction through $\iota_{(\frp)}$' map
$\Hom_{\A_{(\frp)}}(P_{(\frp)},\A_{(\frp)}) \to \Hom_{\A_{(\frp)}}(({^\Pi}\mathcal{O}_{L,S,T}^\times)_{(\frp)},\A_{(\frp)})$. Then Remark \ref{factorsthroughSel} combines with the definition of $\theta$ to imply $\nu(\hat{b}_i^*)=\theta^1(c_i)=\hat{\varphi}_i$ for each $i\in[r]$. It follows that  $N$ is the transpose of the matrix $N(\phi,\{\hat\varphi_i\},\hat{\underline{b}})$ defined in (\ref{matrixN}) and therefore the explicit description given in (\ref{almost2}) implies that $\Phi(\varepsilon_{\hat{\underline{b}}})={\rm Nrd}_A(N^{\rm tr})={\rm Nrd}_A(N)$, as required.
\end{proof}

Upon combining Proposition \ref{ltc prop} with the general result \cite[Th. 3.20(iii)]{bses}, one deduces
\begin{equation}\label{getting there} \delta(\mathcal{A}_{(\frp)})\cdot \Phi(\varepsilon_{\hat{\underline{b}}}) \subseteq {\rm Ann}_{\mathcal{A}_{(\frp)}}(({_{\check\Pi}}\mathcal{S}_S^T(L)/\mathcal{E}_\Phi)_{(\frp)}).\end{equation}
To deduce Theorem \ref{main result2}(iii) from this inclusion (and Proposition \ref{local reduction}), we observe that the commutativity of the square in Lemma \ref{dual complex lemma} combines with 
our choice of $\Phi$ to imply that the lattice $(\mathcal{E}_\Phi)_{(\frp)}$ is disjoint from the kernel of the induced surjection of $\mathcal{A}_{(\frp)}$-modules $(f_{S,S'})_{(\frp)}:({_{\check\Pi}}\mathcal{S}_S^T(L))_{(\frp)} \rightarrow ({_{\check\Pi}}\mathcal{S}_{S'}^T(L))_{(\frp)}$. One thus obtains a surjection of $\mathcal{A}_{(\frp)}$-modules
\begin{equation}\label{getting there2} ({_{\check\Pi}}\mathcal{S}_S^T(L)/\mathcal{E}_\Phi)_{(\frp)} \twoheadrightarrow ({_{\check\Pi}}\mathcal{S}_{S'}^T(L))_{(\frp)}/\bigl((f_{S,S'})_{(\frp)}((\mathcal{E}_\Phi)_{(\frp)})\bigr).\end{equation}
In addition, the exact sequence of $G$-modules (\ref{selmer lemma I}) induces an exact sequence of $\mathcal{A}_{(\frp)}$-modules
%
%
\[ {\rm Tor}^{G}_1(\check\Pi,\Hom_\ZZ(\mathcal{O}_{L,S',T}^\times,\ZZ))_{(\frp)} \xrightarrow{} ({_{\check\Pi}}({\rm Cl}_{S'}^T(L)^\vee))_{(\frp)} \to ({_{\check\Pi}}\mathcal{S}_{S'}^T(L))_{(\frp)}/\bigl((f_{S,S'})_{(\frp)}((\mathcal{E}_\Phi)_{(\frp)})\bigr),\]
and this combines with the surjection (\ref{getting there2}) and containment (\ref{getting there}) to imply the condition in Proposition \ref{local reduction} is satisfied. This therefore completes the proof of Theorem \ref{main result2}.

\end{document}